\documentclass[11pt]{amsart}

\usepackage{amssymb}
\usepackage{amsmath}
\usepackage{amsthm}
\usepackage{amscd}
\usepackage{amsfonts}
\usepackage{ascmac}
\usepackage[usenames]{color}
\usepackage{graphics}

\theoremstyle{plain}
\newtheorem{thm}{Theorem}[section]
\newtheorem{prop}[thm]{Proposition}
\newtheorem{lem}[thm]{Lemma}

\newtheorem{rem}[thm]{Remark}
\newtheorem{prob}[thm]{Problem}

\def\vol{\mathop{\mathrm{Vol}}\nolimits}
\newcommand{\barpartial}{{\overline \partial}}

\newcommand{\mapright}[1]{\smash{\mathop{   \hbox to 0.7cm{\rightarrowfill}}
  \limits^{#1}}}

\renewcommand{\emph}[1]{{\color{red} \it #1}}

\definecolor{orange}{cmyk}{0, 0.7, 1, 0}
\definecolor{light-green}{cmyk}{0.5, 0, 0.5, 0}
\definecolor{light-blue}{cmyk}{0.5, 0, 0, 0}
\definecolor{light-yellow}{cmyk}{0,0,0.6,0}
\definecolor{dark-green}{cmyk}{0.7, 0, 0.7, 0.5}

\title{Volume minimization and Conformally K\"ahler, Einstein-Maxwell geometry}
\author{Akito Futaki and Hajime Ono}
\address{Graduate School of Mathematical Sciences, The University of Tokyo, 3-8-1 Komaba Meguro-ku Tokyo 153-8914, Japan}
\email{afutaki@ms.u-tokyo.ac.jp}
\address{Department of Mathematics, Saitama University, 255 Shimo-Okubo, Sakura-Ku,
Saitama 380-8570, Japan}
\email{hono@rimath.saitama-u.ac.jp}

\date{March 22, 2017}

\begin{document}
\begin{abstract} Let $M$ be a compact complex manifold admitting a K\"ahler structure.
A conformally K\"ahler, Einstein-Maxwell metric (cKEM metric for short) is a 
Hermitian metric $\Tilde g$ on $M$ with constant scalar curvature such that there is a 
positive smooth function $f$
with $g = f^{2} \Tilde g$ being a K\"ahler metric and $f$ being a Killing Hamiltonian potential with respect to $g$.
Fixing a K\"ahler class, we characterize such Killing vector fields whose Hamiltonian function $f$ with
respect to some K\"ahler metric $g$ in the fixed K\"ahler class gives a cKEM metric
$\Tilde g = f^{-2}g$. The characterization is described 
in terms of critical points of certain volume functional. The conceptual idea is similar
to the cases of K\"ahler-Ricci solitons and Sasaki-Einstein metrics in that the derivative of the volume
functional gives rise to a natural obstruction to the existence of cKEM metrics. 
However, unlike the K\"ahler-Ricci soliton case
and Sasaki-Einstein case, the functional is neither
convex nor proper in general, and often has more than one critical points. The last observation matches well with the 
ambitoric examples studied earlier by LeBrun and Apostolov-Maschler. 
%{\color{red}
%We will also see a toric example in which
%the the natural (classical) obstruction vanishes, but the manifold is K-unstable so that no toric cKEM metric exists.
%We will also construct new examples of cKEM metrics and extremal K\"ahler metrics on the
%blow-up of $\mathbf{CP}^2$ at one point.
%}
\end{abstract}

\maketitle

\section{Introduction.}

Let $(M,J)$ be a compact K\"ahler manifold. We call a Hermitian metric
$\Tilde{g}$ on $(M,J)$  
a conformally K\"ahler, Einstein-Maxwell metric
(cKEM metric for short) if it satisfies the following three conditions:

(a) There exists a positive smooth function $f$ on $M$  such that 
$g=f^2\Tilde{g}$  is K\"ahler.

(b) The Hamiltonian vector field $K=J\mathrm{grad}_gf$
 is Killing for both $g$ and $\Tilde{g}$.

(c)  The scalar curvature $s_{\Tilde{g}}$ of $\Tilde g$ is constant.

\noindent
Since the Ricci tensors $\mathrm{Ric}_g$ and 
$\mathrm{Ric}_{\Tilde g}$
of $g$ and $\Tilde g$ are related by
$$ \mathrm{Ric}_{\Tilde g\, 0} = \mathrm{Ric}_{g\, 0} + 2f^{-1} \mathrm{Hess}_0 f$$
where $[\ \ ]_0$ denotes (throughout this paper) the trace free part 
(c.f. (1.161b) in Besse \cite{B}),
the condition (b) is equivalent to

(b') $\mathrm{Ric}_{\Tilde{g}}(J\cdot, J\cdot)=\mathrm{Ric}_{\Tilde{g}}
(\cdot,\cdot)$. 

\noindent
The condition (c) is equivalent to 
\begin{equation}\label{eq:3.1}
s_{\Tilde{g}}=2\left(\dfrac{2m-1}{m-1}\right)f^{m+1}\Delta_g
\left(\frac{1}{f}\right)^{m-1}+s_gf^2=\text{const},
\end{equation}
where $\dim_{\mathbf{C}}M=m$.

If $\Tilde g$ is a cKEM metric and if $\dim_{\mathbf R} = 4$, then one obtains a solution 
$(M, h, F)$ of the following Einstein-Maxwell equation studied in General Relativity 
(see LeBrun \cite{L1}):

(i) $h$ is a Riemannian metric. (In our case $h = \Tilde g$).

(ii) $F$ is a real 2-form.

(iii) $dF = 0$, $d\ast F=0$, $[\mathrm{Ric} + F\circ F]_0 = 0$. Here $(F\circ F)_{jk} = 
F_j{ }^\ell F_{\ell k}$.
(In our case $F^+$ is the K\"ahler form 
$\omega_g$, and $F^- = \frac12 f^{-2} \rho_0(\Tilde g)$ with $\rho_0(\Tilde g)$ the traceless Ricci form of 
$\Tilde g$.)

Except for the constant scalar curvature K\"ahler (cscK for short) metrics in which case $f$ is a 
constant function, not many examples are known.
The most well-known examples may be the conformally K\"ahler Einstein metrics by Page \cite{Page78} 
on the one-point-blow-up
of $\mathbf C\mathbf P^2$, by Chen-LeBrun-Weber \cite{ChenLeBrunWeber} on the two-point-blow-up
of $\mathbf C\mathbf P^2$. Further examples are the ones 
by Apostolov-Calderbank-Gauduchon \cite{ACG16}, \cite{ACG15} on 4-orbifolds and 
 by B\'erard-Bergery \cite{BB82} on $\mathbf P^1$-bundles over Fano K\"ahler-Einstein manifolds. 
In the more recent studies, non-Einstein cKEM examples are constructed by LeBrun \cite{L1}, \cite{L2} 
showing that there are ambitoric examples on $\mathbf C\mathbf P^1 \times 
\mathbf C\mathbf P^1$ and the  one-point-blow-up
of $\mathbf C\mathbf P^2$, 
and by Koca-T{\o}nnesen-Friedman \cite{KT} on ruled surfaces of higher genus.

In \cite{AM}, Apostolov and Maschler initiated a study in the framework similar to the K\"ahler geometry, 
and set the existence problem of cKEM metrics in the Donaldson-Fujiki picture \cite{donaldson97}, \cite{fujiki92}. 
In particular, fixing a 
K\"ahler class, they defined an obstruction to the existence of cKEM metrics in a similar manner to
the K\"ahler-Einstein and cscK cases \cite{futaki83.1}, \cite{futaki83.2}. They further studied the toric surfaces
and showed the equivalence between the existence of cKEM metrics and toric K-stability 
on toric surfaces with convex quadrilateral moment map images, 
extending earlier
works by Legendre \cite{Legendre11} and Donaldson \cite{donaldson02}, \cite{donaldson09GAFA}.
We remark that Lichnerowicz-Matsushima reductiveness theorem for cscK manifolds is also extended to the
cKEM manifolds by Lahdili \cite{Lahdili17} and us \cite{FO_reductive17} independently.

The purpose of the present paper is to study for which Killing vector field we can find a cKEM metric.
We show that, fixing a K\"ahler class, such Killing vector fields are critical points of certain volume functional.
We also show that, for toric manifolds, this idea gives an efficient way to decide which vector fields in the Lie algebra
of the torus can have a solution of the cKEM problem. The idea is similar to the cases of K\"ahler-Ricci solitons
and Sasaki-Einstein metrics, so let us digress to these two cases. A K\"ahler-Ricci soliton is a K\"ahler metric
with its K\"ahler form $\omega \in c_1(M)$ such that there exists a Killing Hamiltonian vector field $X$ in the Lie algebra
$\mathfrak h$ 
of the maximal torus of the automorphism group such that
\begin{eqnarray*}
\rho_\omega &=& \omega + L_{JX }\omega\\
&=& \omega + i\partial\barpartial f_X
\end{eqnarray*}
where $\rho_\omega$ is the Ricci form of $\omega$ and $f_X$ is the Hamiltonian function of $X$.
To find such $X$ that there is a K\"ahler form $\omega$ satisfying the K\"ahler-Ricci soliton equation, let $g$ be an arbitrary
K\"ahler metric with its K\"ahler class $\omega_g \in c_1(M)$, and let $h_g$ be a smooth function such that
$$ \rho_g - \omega_g = \partial\barpartial h_g.$$
Tian and Zhu defined in \cite{TZ02} a functional $F_X : \mathfrak h \to \mathbf R$ by
$$ F_X(Y) = \int_M (JY)(h_g - f_X) e^{f_X} \omega_g^m$$
where $f_X$ is the Hamiltonian function of $X$ with the normalization $\int_M e^{f_X} \omega_g^m = \int_M \omega_g^m$.
$F_X$ is independent of the choice of $\omega_g \in c_1(M)$, and if there exists a K\"ahler-Ricci soliton for $X$ then
$F_X$ vanishes identically. To find such $X$ with vanishing $F_X$, they considered the weighted volume functional
$V : \mathfrak h \to \mathbf R$ defined by
$$ V(Z) = \int_M e^{u_Z} \omega_g^m$$
where $u_Z$ is the Hamiltonian function of $Z \in \mathfrak h$ 
with the the normalization $\int_M u_Z e^{h_g}\omega_g^m = 0$.
They showed that $V$ is independent of $\omega_g$, that $dV_X(Y) = c\, F_X(Y)$ with a constant $c$, that
$V$ is a strictly convex proper function, and that there is a unique minimum $X$. This minimum $X$ is the right choice
to solve the K\"ahler-Ricci soliton equation.

Let us turn to the Sasaki-Einstein metrics. An odd dimensional Riemannian manifold $S$ is said to be a Sasakian manifold
if its Riemannian cone manifold $C(S)$ is a K\"ahler manifold, and a Sasakian manifold $S$ is said to be a Sasaki-Einstein manifold if
$S$ is also an Einstein manifold. A fundamental fact is that $S$ is Sasaki-Einstein if and only if its cone
$C(S)$ is a Ricci-flat K\"ahler manifold, and also if and only if the local leaf spaces of the $1$-dimensional foliation
generated by the Reeb vector field $J(r\frac{\partial}{\partial r})$ have K\"ahler-Einstein metrics where $r$ denotes 
the radial coordinate on the cone $C(S)$. There is an obstruction to the existence of Sasaki-Einstein metrics
similar to K\"ahler-Einstein metrics \cite{FOW}, \cite{BGS}. 
Fixing a holomorphic structure of the cone $C(S)$, 
a natural deformation space of Sasakian structures is
the deformation space of the cone structures of $C(S)$. If such a deformation is given by
$r \mapsto r^{\prime} = re^{\varphi}$ then we have a deformation of the Reeb vector field
$J(r\frac{\partial}{\partial r}) \mapsto J(r^{\prime}\frac{\partial}{\partial r^\prime})$. Thus, this deformation can be 
regarded as a deformation of the Reeb vector fields. Let us define the volume functional 
$V : \mathrm{KCS}(C(S),J) \to \mathbf R$ on the space $\mathrm{KCS}(C(S),J)$ of the K\"ahler cone structures with a fixed 
holomorphic structure $J$ by $V(S,g) = \mathrm{vol}(S,g)$ where $S = \{r=1\}$ is the Sasakian manifold
determined by the K\"ahler cone structure. 
Denote by $G$ the maximal torus of the group of automorphisms commuting with the flow generated by 
$r\frac{\partial}{\partial r}$. 
Martelli-Sparks-Yau \cite{MSY2} showed that
the derivative $dV_{(S,g)}$ 
gives rise to a linear function on $\mathrm{Lie}(G)$ which coincides with the obstruction to the existence of Sasaki-Einstein metrics mentioned above. 
They further showed that, when $S$ is toric (meaning $C(S)$ is toric), 
the volume functional $V$ restricted to the space of toric deformations (meaning deformations of the Reeb vector field
in the Lie algebra of the torus) is a strictly convex proper function, and the unique minimum is the right choice of
the Reeb vector field, i.e. the right choice of the Sasakian structure to solve the Sasaki-Einstein equation
because this minimum assures the vanishing of the obstruction.
For this choice we can always find a Sasaki-Einstein metric \cite{FOW}, \cite{CFO}. See also the survey articles \cite{FO08}, 
\cite{Sparks11}.

To be more precise in our cKEM problem, let $G$ be a maximal torus of a
maximal reductive subgroup of the automorphism group, and take $K \in \mathfrak g := \mathrm{Lie}(G)$.
Let $\omega_0$ be a K\"ahler form,
and $\Omega = [\omega_0] \in H^2_{\mathrm{DR}}(M,\mathbf R)$ be a fixed K\"ahler class. The problem is to
find a $G$-invariant K\"ahler metric $g$ with its K\"ahler form $\omega_g \in \Omega$ such that

(i) $\Tilde g = f^{-2} g$ is a cKEM metric,

(ii) $J\mathrm{grad}_{g} f = K$.

\noindent
Denote by
$\mathcal K_\Omega^G$ the space of $G$-invariant K\"ahler metrics $g$ with $\omega_g
\in \Omega$. For any $(K,a,g)\in \mathfrak g\times \mathbf{R}\times
\mathcal K_\Omega^G$, there exists a unique function $f_{K,a,g}
\in C^\infty(M,\mathbf{R})$
satisfying the following two conditions:
\begin{equation}
\iota_K\,\omega_g=-df_{K,a,g},\ \ 
\int_Mf_{K,a,g}\frac{\omega_g^m}{m!}=a.
\end{equation}
Noting $\min\{f_{K,a,g}\,|\,x \in M\}$ is independent of $g$ with $\omega_g \in \Omega$ (see section 2), 
we put
\begin{eqnarray}
\mathcal P_\Omega^G&:=&\{(K,a)\in \mathfrak g\times \mathbf{R}\,|\, f_{K,a,g} > 0\}, \label{eq:0.1}\\
\mathcal H_\Omega^G&:=&\left\{\Tilde{g}_{K,a}=\frac{1}{f_{K,a,g}^2}g\,
\Big|\,
(K,a)\in \mathcal P_\Omega^G,g\in \mathcal K_\Omega^G\right\}.\label{eq:0.2}
\end{eqnarray}
Hereafter the K\"ahler metric $g$ and its K\"ahler form $\omega_g$ are often identified, and $\omega_g$ is
often denoted by $\omega$.
Fixing $(K,a) \in P^G_{\Omega}$, put
\begin{equation}\label{eq:0.3}
\mathcal H_{\Omega,K,a}^G : =\{\Tilde{g}_{K,a}\,|\,
g \in \mathcal K_{\Omega}^G\} 
\end{equation}
and 
\begin{equation}\label{eq:0.4}
d_{\Omega,K,a}:=
\dfrac{S(\Tilde{g}_{K,a})}{\vol(\Tilde{g}_{K,a})}
=\dfrac
{\displaystyle{\int_Ms_{\Tilde{g}_{K,a}}\left(\frac{1}{f_{K,a,g}}\right)^{2m}
\frac{\omega^m}{m!}}}
{\displaystyle{\int_M\left(\frac{1}{f_{K,a,g}}\right)^{2m}
\frac{\omega^m}{m!}}}.
\end{equation}
Then $d_{\Omega,K,a}$ 
is a constant independent of the choice of $g\in \mathcal K_\Omega^G$ as shown in
\cite{AM}. Let us put further 

\begin{equation}\label{eq:0.5}
\Tilde{\mathcal P}_\Omega^G:=\left\{(K,a)\in \mathcal P_\Omega^G\,\Big|\,
d_{\Omega,K,a}=1\right\}.
\end{equation}

The main result in this paper is the following volume minimization property of cKEM metrics.

\begin{thm}\label{main theorem}
Let $(K,a)\in \Tilde{\mathcal P}_\Omega^G$. Then if there exists a conformally K\"ahler,
Einstein-Maxwell metric $\Tilde{g}_{K,a}\in \mathcal H_{\Omega,K,a}^G$ then $(K,a)$ 
is a critical point of 
$\vol:
\Tilde{\mathcal P}_{\Omega}^G\to \mathbf{R}$
given by $\vol(K,a):=\vol(\Tilde{g}_{K,a})$ for $(K,a)\in 
\Tilde{\mathcal P}_\Omega^G$.
Further, $(K,a)$ 
is a critical point of 
$\vol:
\Tilde{\mathcal P}_{\Omega}^G\to \mathbf{R}$
if and only if 
$\mathfrak F_{\Omega,K,a}^G\equiv 0$.
\end{thm}

\noindent
Here $\mathfrak F_{\Omega,K,a}^G$ is an obstruction to the existence of cKEM metric
defined in \cite{AM}, described as follows.
For $(K,a)\in \mathcal P_\Omega^G$, the following hold:\\

\begin{equation}\label{eq:3.6}
c_{\Omega,K,a}:=\dfrac
{\displaystyle{\int_Ms_{\Tilde{g}_{K,a}}\left(\frac{1}{f_{K,a,g}}\right)^{2m+1}
\frac{\omega^m}{m!}}}
{\displaystyle{\int_M\left(\frac{1}{f_{K,a,g}}\right)^{2m+1}
\frac{\omega^m}{m!}}}
\end{equation}
is a constant independent of the choice of $g\in \mathcal K_\Omega^G$.
Then 
\begin{equation}\label{eq:3.7}
\mathfrak F_{\Omega,K,a}^G:\mathfrak g\to \mathbf R,\ \ 
\mathfrak F_{\Omega,K,a}^G(H):=\int_M\left(\dfrac{s_{\Tilde{g}_{K,a}}-c_{\Omega,K,a}}
{f_{K,a,g}^{2m+1}}\right)f_{H,b,g}\dfrac{\omega^m}{m!}
\end{equation}
is a linear function independent of the choice of $(g,b)\in
\mathcal K_\Omega^G\times \mathbf{R}$. Obviously, if there exists a
constant scalar curvature metric in $\mathcal H_{\Omega,K,a}$,
then $\mathfrak F_{\Omega,K,a}^G$ is identically zero.
For terminological convenience we call $\mathfrak F_{\Omega,K,a}^G$ the cKEM-Futaki invariant.
A merit of Theorem \ref{main theorem} is to give a systematic computation of the cKEM-Futaki invariant.

This paper is organized as follows.
In section 2 we give a proof of Theorem \ref{main theorem}.
In section 3 we give examples of non K\"ahler cKEM metrics. This is an
extension of LeBrun's construction (\cite{L1}) on $\mathbf C\mathbf P^1 \times \mathbf C\mathbf P^1$
to $\mathbf C\mathbf P^1 \times M$ for higher dimensional $M$'s. 
In section 4, we use Maxima (a descendant of Macsyma) to compute the cKEM-Futaki invariant of 
$\mathbf C\mathbf P^1 \times \mathbf C\mathbf P^1$, the blow-up of $\mathbf C\mathbf P^2$ at one point
and other Hirzebruch surfaces.

\section{Proof of Theorem \ref{main theorem}}

Let $M$ be a compact $n$-manifold with $n\ge 3$.
Let 
$\mathrm{Riem}(M)$ denote the set of all Riemannian metrics on $M$, 
$ \mathrm{Ric}_g$ the Ricci tensor of $g$, 
$s_g$ the scalar curvature of $g$, and
$dv_g$ the volume form of $g$.
The normalized Einstein-Hilbert functional $EH : \mathrm{Riem}\,(M)\to \mathbf{R}$ is defined by
$$EH(g):=\dfrac{S(g)}{(\vol(g))^{\frac{n-2}{n}}}$$
where 
$S(g) = \int_Ms_g\,dv_g$ is the total scalar curvature and $\vol(g) = \int_M\,dv_g$ is the 
volume of $g$. 

The following first variation formulae are standard, and can be found in \cite{B}.
Let $g_t$ be a smooth family of Riemannian metrics such that
$g_0=g$ and $\dfrac{d}{dt}_{|t=0}g_t=h$. Then 
\begin{equation}\label{eq:1.1}
\dfrac{d}{dt}_{\vert_{t=0}}S(g_t)=
\int_M\left\langle \frac{s_g}{2}g-\mathrm{Ric}_g,h\right\rangle_gdv_g.
\end{equation}

Let $f_t$ be a smooth family of positive
functions such that $f_0=1,\dfrac{d}{dt}_{|t=0}f_t=\phi$. Then
$$
\dfrac{d}{dt}_{|t=0}S(f_tg)=\int_M\left\langle
\frac{s_g}{2}g-\mathrm{Ric}_g,\phi g\right\rangle dv_g=
\frac{n-2}{2}\int_Ms_g\phi dv_g
$$
and
$$
\dfrac{d}{dt}_{|t=0}\vol(f_tg)=\dfrac{d}{dt}_{|t=0}\int_Mf_t^{\frac{n}{2}}dv_g
=\frac{n}{2}\int_M\phi dv_g.
$$
Therefore
\begin{equation}\label{eq:1.2}
\dfrac{d}{dt}_{|t=0}EH(f_tg)=
\dfrac{n-2}{2(\vol (g))^{\frac{n-2}{n}}}
\int_M\left(s_g-\dfrac{S(g)}{\vol(g)}\right)\phi dv_g.
\end{equation}
By \eqref{eq:1.2}, if $s_g$ is a constant then g is a critical point of
$EH$ restricted to the conformal class of $g$.

Let $f_t$ be a smooth family of positive functions
such that $f_0=f,\dfrac{d}{dt}_{|t=0}f_t=\phi$.
Denote $\dfrac{1}{f_t^2}g=:\Tilde{g}_t,\Tilde{g}=:\Tilde{g}_0$.
Then
$$
\dfrac{d}{dt}_{|t=0}S(\Tilde{g}_t)=\int_M\left\langle
\frac{s_{\Tilde{g}}}{2}-\mathrm{Ric}_{\Tilde{g}},\frac{-2\phi}{f}\Tilde{g}
\right\rangle_{\Tilde{g}}dv_{\Tilde{g}}=(2-n)\int_M\dfrac{s_{\Tilde{g}}\phi}
{f^{n+1}}dv_g
$$
and
$$
\dfrac{d}{dt}_{|t=0}\vol(\Tilde{g}_t)=
\dfrac{d}{dt}_{|t=0}\int_M\dfrac{1}{f_t^n}dv_g
=-n\int_M\dfrac{\phi}{f^{n+1}}dv_g.
$$
Therefore
\begin{equation}\label{eq:1.3}
\dfrac{d}{dt}_{|t=0}EH(\Tilde{g}_t)=
\dfrac{2-n}{(\vol(\Tilde{g}))^{\frac{n-2}{n}}}
\int_M\dfrac{s_{\Tilde{g}}-\frac{S(\Tilde{g})}{\vol(\Tilde{g})}}{f^{n+1}}\,\phi\, dv_g.
\end{equation}

%\section{Strongly Hermitian solutions of the Einstein-Maxwell equation}
%\section{Conformally K\"ahler, Einstein-Maxwell metrics}

We wish to apply the formula (\ref{eq:1.3}) to the existence problem of cKEM metrics.
Let us recall the situation explained in the introduction. Let $M$ be a compact K\"ahler manifold
of complex dimension $m$ so that $n = 2m$.
Fix a compact group $G\subset \mathrm{Aut}_r(M,J)$ in the group of
reduced automorphisms of $(M,J)$, and consider a fixed K\"ahler class
$\Omega$ on $(M,J)$. Denote by
$\mathcal K_\Omega^G$ the space of $G$-invariant K\"ahler metrics $\omega$
in $\Omega$. For any $(K,a,g)\in \mathfrak g\times \mathbf{R}\times
\mathcal K_\Omega^G$, there exists unique function $f_{K,a,g}
\in C^\infty(M,\mathbf{R})$
satisfying the following two conditions:
\begin{equation}\label{eq:3.2}
\iota_K\omega=-df_{K,a,g},\ \ 
\int_Mf_{K,a,g}\frac{\omega^m}{m!}=a.
\end{equation}
By \eqref{eq:1.3} and \eqref{eq:3.2}, 
it is easy to see that $f_{K,a,g}$ has the following properties:
\begin{equation}\label{eq:3.3}
f_{K+H,a+b,g}=f_{K,a,g}+f_{H,b,g}
\end{equation}
\begin{equation}\label{eq:3.4}
f_{0,a,g}=\dfrac{a}{\vol(M,\omega)}
\end{equation}
\begin{equation}\label{eq:3.5}
f_{CK,Ca,g}=Cf_{K,a,g}
\end{equation}
If $K\not=0$, 
$$
m_{K,\Omega}:=\min\{f_{K,0,g}\,|\, x\in M\}<0.
$$
Note here that, by Lemma $1$ of \cite{AM}, 
$\min\{f_{K,0,g}\,|\, x\in M\}$ is independent of the choice of 
$g\in \mathcal K_\Omega^G$. Hence $f_{K,a,g}$ is positive
if $a>l_{K,\Omega}:=-m_{K,\Omega}\vol(M,\omega)$, by \eqref{eq:3.3} and
\eqref{eq:3.4}.

The set $\mathcal P_\Omega^G$ defined in (\ref{eq:0.1}) in the introduction then can be expressed as 
\begin{equation*}
\mathcal P_\Omega^G =\{(K,a)\in \mathfrak g\times \mathbf{R}\,|\,
a>l_{K,\Omega}\}.
\end{equation*}
With the notation of $\mathcal H_{\Omega,K,a}^G$ given in (\ref{eq:0.3}), the main problem 
of the cKEM metrics we wish to consider is the following : 
\begin{prob}\label{prob1}
Is there a cKEM metric in $\mathcal H_{\Omega,K,a}^G$ for $(K,a)\in
\mathcal P_\Omega^G$?
\end{prob}

By \eqref{eq:3.2}, the Hamiltonian function $f_{K,a,g}$
of $K$ with respect to the K\"ahler metric $g$ is Killing for both
$g$ and $\Tilde{g}_{K,a} = f_{K,a,g}^{-2} g$. So the problem above is equivalent to 
the existence of the constant scalar curvature metric in
$\mathcal H_{\Omega,K,a}^G$.
As was explained in the introduction, Apostolov and Maschler \cite{AM} introduced an obstruction
$\mathfrak F_{\Omega,K,a}^G:\mathfrak g\to \mathbf R$ 
to this problem, see (\ref{eq:3.7}).

Fix $(K,a)\in \mathcal P_\Omega^G$.
By Corollary $2$ of \cite{AM}, the functional $S$ is constant
on $\mathcal H_{\Omega,K,a}^G$. Similarly by the proof of Lemma $1$ of
\cite{AM}, $\vol$ is constant on $\mathcal H_{\Omega,K,a}^G$.
Thus we obtain the following proposition.

\begin{prop} With the notations as above, we have the following.
\begin{enumerate}
\item For $(K,a)\in \mathcal P_\Omega^G$,
\begin{equation}\label{eq:4.1}
d_{\Omega,K,a}:=
\dfrac{S(\Tilde{g}_{K,a})}{\vol(\Tilde{g}_{K,a})}
=\dfrac
{\displaystyle{\int_Ms_{\Tilde{g}_{K,a}}\left(\frac{1}{f_{K,a,g}}\right)^{2m}
\frac{\omega^m}{m!}}}
{\displaystyle{\int_M\left(\frac{1}{f_{K,a,g}}\right)^{2m}
\frac{\omega^m}{m!}}}
\end{equation}
is a constant independent of the choice of $g\in \mathcal K_\Omega^G$.\\
\item The function
\begin{equation}\label{eq:4.2}
V:\mathcal P_\Omega^G\to \mathbf{R},\ \ (K,a)\mapsto EH(\Tilde{g}_{K,a})
\end{equation}
is well-defined.
\end{enumerate}
\end{prop}
\begin{thm} We have the following results.\\
(a) If there 
 exists a cKEM metric $\Tilde{g}_{K,a}\in \mathcal H_{\Omega,K,a}^G$, then
$(K,a)$ is a critical point of $V:\mathcal P_{\Omega}^G\to \mathbf{R}$. \\
(b) If $(K,a)$ is a critical point of $V:\mathcal P_{\Omega}^G\to \mathbf{R}$, 
$\mathfrak F_{\Omega,K,a}^G$ vanishes identically.
\end{thm}
\begin{proof}
The first statement (a)  is trivial by the first variation of the Einstein-Hilbert functional
\eqref{eq:1.3}.
To prove the second statement (b), 
suppose that $(K,a)\in \mathcal P_\Omega^G$ is a critical point of $V$.
For $H\in \mathfrak g$, 
\begin{eqnarray}
0 &=& 
\dfrac{d}{dt}_{|t=0}V(K+tH,a) \nonumber\\
&=&
\dfrac{2-2m}{\vol(M,\Tilde{g}_{K,a})^{\frac{m-1}{m}}}
\int_M\left(\dfrac{s_{\Tilde{g}_{K,a}}-d_{\Omega,K,a}}
{f_{K,a,g}^{2m+1}}\right)f_{H,0,g}\dfrac{\omega^m}{m!}  \label{eq:4.3} 
\end{eqnarray}
by \eqref{eq:1.3}, \eqref{eq:3.3} and \eqref{eq:3.5}.
Note here that, in general, $c_{\Omega,K,a}$ does not coincide with $d_{\Omega,K,a}$.
So \eqref{eq:4.3} does not mean the vanishing of cKEM-Futaki invariant $\mathfrak F_{\Omega,K,a}^G$ in (\ref{eq:3.7}).
However if there exists a cKEM metric in $\mathcal H_{\Omega,K,a}$, 
$c_{\Omega,K,a}=d_{\Omega,K,a}$ holds, i.e. $c_{\Omega,K,a}-d_{\Omega,K,a}$
is an obstruction to the existence of cKEM metric.
This obstruction can be represented as the $\mathbf{R}$-direction
first variation of $V$
as follows:
\begin{eqnarray}\label{eq:4.4}
0&=&\dfrac{d}{dt}_{|t=0}V(K,a+t) \\
&=&\dfrac{2-2m}{\vol(M,\Tilde{g}_{K,a})
^{\frac{2m-1}{m}}}(c_{\Omega,K,a}-d_{\Omega,K,a})
\int_M\dfrac{1}{f_{K,a,g}^{2m+1}}\dfrac{\omega^m}{m!}.\nonumber
\end{eqnarray}
Therefore if $(K,a)\in \mathcal P_\Omega^G$ is a critical point of 
$V$, then $ \mathfrak F_{\Omega,K,a}^G\equiv 0$.
\end{proof}

%\underline{\bf Volume minimizing property of cKEM metrics}

%\vspace{6mm}

There exists a cKEM metric in $\mathcal H_{\Omega,K,a}^G$
if and only if there exists a cKEM metric in
$\mathcal H_{\Omega,CK,Ca}^G$ for $C>0$.
Therefore Problem \ref{prob1} is equivalent to the following
problem.

\begin{prob}
Is there a cKEM metric in $\mathcal H_{\Omega,K,a}^G$ for
$[(K,a)]\in \mathcal P_\Omega^G/\mathbf{R}_{>0}$?

\end{prob}

So we wish to have good representatives of elements in
$\mathcal P_\Omega^G/\mathbf{R}_{>0}$.
This is the motivation to define $\Tilde{\mathcal P}_\Omega^G$ as in the introduction, see 
(\ref{eq:0.5}).
Since $d_{\Omega,CK,Ca}=C^2d_{\Omega,K,a}$ for
$C>0$, 
$$
\Tilde{\mathcal P}_\Omega^G\simeq \mathcal P_\Omega^G/\mathbf{R}_{>0},\ \ 
(K,a)\mapsto [(K,a)]
$$
is bijective. 

\begin{proof}[\bf{Proof of Theorem \ref{main theorem}}]
Let $(K(t),a(t)),\ t\in (-\varepsilon,
\varepsilon)$ be a smooth curve in $\Tilde{\mathcal P}_\Omega^G$
such that $(K(0),a(0))=(K,a),(K'(0),a'(0))=(H,b)$.
Then 
$$
S(\Tilde{g}_{K(t),a(t)})=\vol(\Tilde{g}_{K(t),a(t)})
$$
holds for any $t\in (-\varepsilon,
\varepsilon)$. By differentiating this equation at $t=0$, we have
\begin{equation}\label{eq:4.5}
(m-1)\int_M
\dfrac{s_{\Tilde{g}_{K,a}}f_{H,b,g}}{f_{K,a,g}^{2m+1}}
\dfrac{\omega^m}{m!}=m
\int_M
\dfrac{f_{H,b,g}}{f_{K,a,g}^{2m+1}}
\dfrac{\omega^m}{m!}.
\end{equation}
Hence 
\begin{equation}\label{eq:4.6}
\begin{split}
\mathfrak F_{\Omega,K,a}^G(H)&=\left(
\dfrac{m}{m-1}-c_{\Omega,K,a}\right)\int_M\dfrac{f_{H,b,g}}
{f_{K,a,g}^{2m+1}}\dfrac{\omega^m}{m!}\\
&=\left(\dfrac{1}{m-1}-\dfrac{c_{\Omega,K,a}}{m}\right)
\dfrac{d}{dt}_{|t=0}\vol(M,\Tilde{g}_{K(t),a(t)})
\end{split}
\end{equation}
holds by \eqref{eq:3.7} and \eqref{eq:4.5}.
If there exists a cKEM metric in
$\mathcal H_{\Omega,K,a}^G$ for $(K,a)\in \Tilde{\mathcal P}_\Omega^G$,
then 
$$c_{\Omega,K,a}=d_{\Omega,K,a}=1$$ 
and
$$\mathfrak F_{\Omega,K,a}^G(H)=0.$$
 Therefore
$$\dfrac{d}{dt}_{|t=0}\vol(M,\Tilde{g}_{K(t),a(t)})=0.$$
\end{proof}

\section{Examples of non K\"ahler cKEM metrics}

In this section, we construct compact, non-K\"ahler examples of 
cKEM metrics of dimension greater than two.

Let $g_1$ be an $S^1$-invariant metric on $\mathbf{C}\mathbf P^1$ and
$g_2$ a K\"ahler metric  with $s_{g_2}=c$
on an $(m-1)$-dimensional compact complex manifold $M$. 
The $S^1$-invariant metric $g_1$ can be written in the action-angle coordinates
$(t,\theta)\in (a,b)\times (0,2\pi]$ as
$$
g_1=\dfrac{dt^2}{\Psi(t)}+\Psi(t)d\theta^2
$$
for some smooth function $\Psi (t)$. The Hamiltonian function of the generator
of the $S^1$-action is $t$\footnote{In this section, instead of sliding the Hamiltonian
function by constant, we move the interval $(a,b)$ (moment image).}.
Suppose that $a>0$.
We will look for $\Psi$ such that the Hermitian metric $g/t^2$, where $g=g_1+g_2$,
on $\mathbf{C}\mathbf P^1\times M$ has constant scalar curvature.
Since $\partial /\partial \theta$ is Killing both for $g$ and $g/t^2$,
if we find such $\Psi$, $g/t^2$ is a non-K\"ahler cKEM metric.
The scalar curvature of $g_1$ is given by
$$
s_1=\Delta_{g_1}\log \Psi=-\Psi''(t).
$$
Thus the scalar curvature of $g$ is given by
$$
s=s_1+s_2=c-\Psi''(t).
$$
We want to arrange that the scalar curvature of $h=t^{-2}g$ is $d$, which is
to say that
\begin{equation}\label{eq:5.1}
d=s_h=2\left(\dfrac{2m-1}{m-1}\right)t^{m+1}\Delta_g(t^{1-m})+
(c-\Psi''(t))t^2.
\end{equation}
We may rewrite this as
$$
c-\Psi''=\dfrac{d}{t^2}-2\left(\dfrac{2m-1}{m-1}\right)t^{m-1}\Delta_{g_1}
(t^{1-m})
$$
since the Hessian of $t$ is trivial in the $M$-directions.
Since
$$
\Delta_{g_1}\left(\dfrac{1}{t^{m-1}}\right)=(m-1)\left(
\dfrac{\Psi}{t^m}
\right)',
$$
the equation \eqref{eq:5.1} reduces to the ODE
$$
c-\Psi''=\dfrac{d}{t^2}-2(2m-1)\dfrac{\Psi'}{t}+2m(2m-1)\dfrac{\Psi}{t^2},
$$
or equivalently
\begin{equation}\label{eq:5.2}
t^2\Psi''-2(2m-1)t\Psi'+2m(2m-1)\Psi=ct^2-d.
\end{equation}
The general solution of the equation \eqref{eq:5.2} is
$$
\Psi(t)=At^{2m}+Bt^{2m-1}+\dfrac{c}{2(m-1)(2m-3)}t^2-\dfrac{d}{2m(2m-1)}.
$$
Now, in order to get a metric on $S^2$, we need to impose the boundary conditions
that
$$
\Psi(a)=\Psi(b)=0,\Psi'(a)=-\Psi'(b)=2,\ \ \Psi(t)>0\ (\text{on } (a,b)).
$$
The first four conditions reduce to a simultaneous linear equation for $A,B,c$ and $d$.
The solution is
\begin{eqnarray*}%\label{eq:5.3}
A_{a,b,m}&=& (a^2b^2(a+b)\{2\left( a-b\right) \left( {b}^{2m-2}+{a}^{2m-2}\right) m\\
&\ \ &\qquad + 3{b}^{2m-1}-a{b}^{2m-2}+{a}^{2m-2}b
-3{a}^{2m-1}\})/E_{a,b,m}, \\
%\end{eqnarray*}
%\begin{eqnarray*}%\label{eq:5.4}
B_{a,b,m}&=& (2a^2b^2(a+b)\{(\left( b-a\right)\left( {b}^{2m-1}+{a}^{2m-1}\right) )m\\
&\ \ &\qquad - \left( {b}^{m}-{a}^{m}\right) \left( {b}^{m}+{a}^{m}\right) \})/E_{a,b,m},\\
%\end{eqnarray*}
%\begin{eqnarray*}%\label{eq:5.5}
\frac{c_{a,b,m}}{2(m-1)(2m-3)}&=&({b}^{2m}\left( 2{a}^{2m}{b}^{2}-2{a}^{2m+2}\right) m\\
&\ \ &\qquad -{a}^{2}{b}^{4m}+{b}^{2m}\left( {a}^{2m+2}-{a}^{2m}{b}^{2}\right) +{a}^{4m}{b}^{2})/E_{a,b,m},\\
%\end{eqnarray*}
%and
%\begin{eqnarray*}%\label{eq:5.6}
\frac{d_{a,b,m}}{2m(2m-1)}&=&{b}^{2m}\left( 2{a}^{2m+1}{b}^{3}-2{a}^{2m+3}b\right) m\\
&\ & -{a}^{4}{b}^{4m}+{a}^{4m}{b}^{4}+{b}^{2m}\left( 2{a}^{2m+3}b-2{a}^{2m+1}{b}^{3}\right))/E_{a,b,m},
\end{eqnarray*}
where
\begin{equation}\label{eq:5.7}
\begin{split}
E_{a,b,m}=&(2{a}^{2m}{b}^{2m}{\left( b-a\right) }^{2}\left( b+a\right) )m^2
-3{a}^{2m}{b}^{2m}{\left( b-a\right) }^{2}\left( b+a\right) m\\
&-\left( {b}^{m}-{a}^{m}\right) \left( {b}^{m}+{a}^{m}\right) \left( {a}^{3}{b}^{2m}-{a}^{2m}{b}^{3}\right) .
\end{split}
\end{equation}

If we set
$$
\Psi_{a,b,m}(t)=A_{a,b,m}t^{2m}
+B_{a,b,m}t^{2m-1}+\dfrac{c_{a,b,m}}{2(m-1)(2m-3)}t^2
-\dfrac{d_{a,b,m}}{2m(2m-1)},
$$
we have $\Psi_{a,b,m}>0$ on $(a,b)$ by the following lemma for any $m$ and
$0<a<b$.
Therefore, for
$m\ge 2$, 
$$g_{a,b,m}=\frac{dt^2}{\Psi_{a,b,m}(t)}+\Psi_{a,b,m}(t)d\theta^2$$
defines a
metric on $\mathbf C\mathbf P^1$. 

\begin{lem}
Let $m\ge 2$ be an integer and $0<a<b$.
If a real valued function
$$
f(t)=\alpha t^{2m}+\beta t^{2m-1}+\gamma t^2+\delta
$$
satisfies the boundary conditions
$$f(a)=f(b)=0,\ \ f'(a),-f'(b)>0,$$
then $f>0$ on $(a,b)$.
\end{lem}
\begin{proof}
Suppose that there exists $c\in (a,b)$ such that $f(c)\le 0$. Then, by the boundary
condition, there exist at least three critical points of $f$ in $(a,b)$.
On the other hand, since
$$
\frac{f'(t)}{t}=2m\alpha t^{2m-2}+(2m-1)\beta t^{2m-3}+2\gamma,
$$
$$
\left(\frac{f'(t)}{t}\right)'=t^{2m-4}\left\{2m(2m-2)\alpha t+(2m-1)(2m-3)\beta\right\},
$$
$\dfrac{f'}{t}$ has at most two zeros in $(a,b)$. This is a contradiction.
\end{proof}

Moreover if
$g_2$ is a K\"ahler metric with $s_{g_2}=c_{a,b,m}$ on an $(m-1)$-dimensional
compact complex manifold $M$,
$$
h_{a,b,m}(g_2)=\frac{1}{t^2}(g_{a,b,m}+g_2)
$$
is an $S^1$-invariant
cKEM metric with $s_{h_{a,b,m}(g_2)}=d_{a,b,m}$ on $\mathbf{CP}^1\times
M$.

For simplicity, we now put $b=a+1$. 
Let
$$
\mathcal C_m=\{c_{a,a+1,m}\,|\, a>0\}.
$$
Since
$$
\lim_{a\to +0}c_{a,a+1,m}=\infty,\ \ \lim_{a\to \infty}c_{a,a+1,m}=8m-8,
$$
we see that $(8m-8,\infty)\subset \mathcal C_m$.
Hence we have proved the following.

\begin{thm}
Let $c>8m-8$. Then there exists $a>0$ such that
for any K\"ahler metric $g_2$ with $s_{g_2}=c=c_{a,a+1,m}$ on an $(m-1)$-dimensional
compact complex manifold $M$, $h_{a,a+1,m}(g_2)$ is an $S^1$-invariant
cKEM metric on $\mathbf{CP}^1\times M$.
On the other hand, if $c\not\in \mathcal C_m$, for any $a>0$ and
$g_2$, $h_{a,a+1,m}(g_2)$ is not a cKEM metric. 
\end{thm}

Note here that, for small $m$, e.g. $m = 2, 3, 4, 5$, we can directly
confirm that $\mathcal C_m=(8m-8,\infty)$.
Hence in such cases, $c\not\in \mathcal C_m$ if and only if $c\le 8m-8$.

\noindent
This theorem extends the case when $M = \mathbf C\mathbf P^1$ due to LeBrun \cite{L1}.

\section{Computations in the case of toric surfaces}

Let $(M,J,g)$ be an $m$-dimensional compact toric K\"ahler manifold.
We denote by
$\Delta\subset \mathbf{R}^m$,
$u$ and $\text{{\bf H}}^u$ the moment polytope, the symplectic potential  
and the inverse $\mathrm{Hess}(u)^{-1}$ of the Hessian of $u$ respectively. 
Then by the equation $(22)$ in \cite{AM},
\begin{equation}\label{eq:6.1}
\dfrac{s_{\Tilde{g}_{K,a}}}{f_{K,a,g}^{2m}}=
-f_{K,a,g}\sum_{i,j=1}^m\left(
\dfrac{1}{f_{K,a,g}^{2m-1}}\text{{\bf H}}^u_{ij}
\right)_{,ij}
\end{equation}
holds. Since $f_{K,a,g}$ is an affine linear function of action 
coordinates,
the equation $(30)$ in \cite{AM} implies
\begin{equation}\label{eq:6.2}
\int_M\dfrac{s_{\Tilde{g}_{K,a}}}{f_{K,a,g}^{2m}}
\dfrac{\omega^m}{m!}=
\dfrac{2(2\pi)^m}{m!}\int_{\partial \Delta}\dfrac{1}{f_{K,a,g}^{2m-2}}
\,d\sigma.
\end{equation}
On the other hand
$$
\int_M\dfrac{1}{f_{K,a,g}^{2m}}
\dfrac{\omega^m}{m!}=
\dfrac{(2\pi)^m}{m!}\int_{\Delta}\dfrac{1}{f_{K,a,g}^{2m}}
\,d\mu
$$
holds. Therefore $V$ is given by
\begin{equation}\label{eq:6.3}
V(K,a)=\dfrac{4\pi}{(m!)^{\frac{1}{m}}}
\dfrac{\displaystyle{\int_{\partial \Delta}\dfrac{1}{f_{K,a,g}^{2m-2}}
\,d\sigma}}{\displaystyle{\left(\int_{\Delta}\dfrac{1}{f_{K,a,g}^{2m}}
\,d\mu\right)^{\frac{m-1}{m}}}}.
\end{equation}
 
In what follows, the coordinates of the moment map image of toric K\"ahler surfaces will
be denoted by $(\mu_1, \mu_2)$.

%\vspace{6mm}

\subsection{$\mathbf{C}\mathbf P^2$ case}

%\vspace{4mm}

In this case, up to scale, $\Delta$ is the convex hull of the
three points $(0,0),(1,0)$ and $(0,1)$. An affine linear function
$f = a\mu_1 + b\mu_2 + c$ is positive on $\Delta$ if and only if $c,a+c,b+c>0$.

Since
\begin{align*}
\int_{\partial \Delta}\dfrac{d\sigma}{(a\mu_1+b\mu_2+c)^2}&=
\dfrac{3c+a+b}{c(a+c)(b+c)},\\
\int_\Delta\dfrac{d\mu_1 d\mu_2}{(a\mu_1+b\mu_2+c)^4}&=\dfrac{3c^2+2(a+b)c+ab}{6c^2(a+c)^2(b+c)^2},
\end{align*}
we have
\begin{eqnarray}\label{eq:6.4}
V(a,b,c)&:=&2\sqrt{2}\pi\dfrac{\displaystyle{\int_{\partial \Delta}\dfrac{1}{(a\mu_1+b\mu_2+c)^2}
\,d\sigma}}{\displaystyle{\left(\int_{\Delta}\dfrac{1}{(a\mu_1+b\mu_2+c)^4}
\,d\mu_1 d\mu_2\right)^{\frac12}}}\\
&=& 
\dfrac{4\sqrt{3}\pi(3c+a+b)}{\sqrt{3c^2+2(a+b)c+ab}} \nonumber
\end{eqnarray}
and
\begin{equation}\label{eq:6.5}
\dfrac{\partial V}{\partial c}(a,b,c)
=-\dfrac{4\sqrt{3}\pi(a^2-ab+b^2)}{(3c^2+2(a+b)c+ab)^{\frac32}}.
\end{equation}
Therefore, by \eqref{eq:4.4} and \eqref{eq:6.5},
$\mathcal H_{\Omega,(a,b),c}^{T^2}$ admits cKEM equation 
only if $(a,b)=(0,0)$.
(Note that the notation $\mathcal H_{\Omega,(a,b),c}^{T^2}$ is a replacement of 
the previous notation $\mathcal H_{\Omega,K,a}^G$; 
 this is allowed because $(K,a)$ is determined by $((a,b),c)$. )
Hence, we obtain the following result.
\begin{prop}
Up to constant multiple, the Fubini-Study metric is the only
 $T^2$-invariant cKEM metric on
$\mathbf{C}\mathbf P^2$. 
\end{prop}
Note that it is possible to prove this proposition without using the volume minimization.
In fact, (iii) in the introduction can be written as 
$$ \mathrm{Ric}_0 = - F^+\circ F^-$$
with $F^+$ and $F^-$ are self-dual and anti-self dual harmonic forms (c.f. \cite{L1}). 
Thus, on $\mathbf C\mathbf P^2$, we have 
$F^- = 0$, and the cKEM metric is Einstein. But compact Hermitian Einstein 
4-manifolds are classified by LeBrun \cite{LeBrun12} to be either the Fubini-Study metric on 
$\mathbf C\mathbf P^2$, the Page metric or Chen-LeBrun-Weber metric. Thus the Fubini-Study metric
is the only cKEM metric on $\mathbf C\mathbf P^2$. Orbifold cKEM metrics on weighted projective planes are also classified by Apostolov-Maschler \cite{AM}, Theorem 4.

%\vspace{6mm}

\subsection{$\boldsymbol{\mathbf{C}\mathbf P^1\times \mathbf{C}\mathbf P^1}$ case}

Let 
$\Delta_p$ be the convex hull of $(0,0),(p,0),(p,1),(0,1)$, where 
$p\ge 1$.
An affine linear function $f = a\mu_1+b\mu_2+c$ is positive on
$\Delta_p $ if and only if  $c,pa+c,pa+b+c,b+c>0$.
We denote
\begin{align*}
&\mathcal P_p:=\{(a,b,c)\in \mathbf{R}^3\,|\, c,pa+c,pa+b+c,b+c>0\},\\
&\mathcal P_p(1):=\mathcal P_p\cap \{(a,b,c)\,|\, b+pa+2c=1\}.
\end{align*}
Note that this choice of $b+pa+2c=1$ can be replaced by any other affine linear function
giving a slice of $\mathcal P_p$.
We chose this simply because it gives a simpler computation.
For $(a,b,c)\in \mathcal P_p$, we define $s_p(a,b,c)$ and $v_p(a,b,c)$ by
\begin{eqnarray*}
&\ &\int_{\partial \Delta_p}\dfrac{d\sigma}{(a\mu_1+b\mu_2+c)^2} \\
&\ &=\frac{1}{c\left( b+c\right) }+\frac{p}{\left( b+c\right) \left( b+pa+c\right) }+\frac{1}{\left( b+pa+c\right) \left( pa+c\right) }+\frac{p}{\left( pa+c\right) c}\\
&\ &=\frac{\left( 2ac+{a}^{2}+ab\right) {p}^{2}+\left( 2{c}^{2}+2\left( a+b\right) c+ab+{b}^{2}\right) p+2{c}^{2}+2bc}{\left( pa+c\right) \left( b+pa+c\right) \left( b+c\right) c}\\
&\ &=\dfrac{s_p(a,b,c)}{\left( pa+c\right) \left( b+pa+c\right) \left( b+c\right) c}
\end{eqnarray*}
and
\begin{eqnarray*}
&\ &\int_{\Delta_p}\dfrac{d\mu_1d\mu_2}{(a\mu_1+b\mu_2+c)^4}\\
&\ &=p\{
\left( 2{a}^{3}c+a^3{b}\right) {p}^{3}+\left( 8{a}^{2}{c}^{2}+8a^2{b}c+2{a}^{2}{b}^{2}\right) {p}^{2}\\
&\ &\quad +\left( 12a{c}^{3}+18ab{c}^{2}+8{a}b^2c+{a}b^3\right) p\\
&\ &\quad +6{c}^{4}+12b{c}^{3}+8{b}^{2}{c}^{2}+2{b}^{3}c\}/(6\left( pa+c\right)^2 \left( b+pa+c\right)^2 \left( b+c\right)^2 c^2)\\
&\ &=\dfrac{v_p(a,b,c)}{6\left( pa+c\right)^2 \left( b+pa+c\right)^2 \left( b+c\right)^2 c^2}.
\end{eqnarray*}
Then, for $(a,b,c)\in \mathcal P_p(1)$, we have
\begin{equation}\label{eq:6.6}
\begin{split}
(V_p(a,b))^2 
&:=
\dfrac{48\pi^2s_p(a,b,(1-b-pa)/2)^2}{v_p(a,b,(1-b-pa)/2)}\\
&=
-\frac{96\pi^2{\left( {a}^{2}{p}^{3}-{a}^{2}{p}^{2}-{b}^{2}p-p+{b}^{2}-1\right) }^{2}}{p\left( {a}^{4}{p}^{4}-2{a}^{2}{b}^{2}{p}^{2}+2{a}^{2}{p}^{2}+{b}^{4}+2{b}^{2}-3\right) }
\end{split}
\end{equation}
Note here that, in general,
the function $V$ on $\mathcal P_\Omega^G$ is scale invariant, that 
is $V(dK,da)=V(K,a)$ for any $(K,a)\in \mathcal P_{\Omega}^G$ and $d>0$.
Hence if $(a,b)$ is a critical point of $V_p(a,b)$ and $(a,b,(1-b-pa)/2)
\in \mathcal P_p(1)$, then cKEM-Futaki invariant for  $(a,b,(1-b-pa)/2)$
vanishes. The derivatives are computed as follows:

\begin{align*}
\dfrac{\partial V_p^2}{\partial a}&=
-\frac{768\pi^2ap\left( {a}^{2}{p}^{3}+{b}^{2}p-p-2{b}^{2}+2\right) \left( {a}^{2}{p}^{3}-{a}^{2}{p}^{2}-{b}^{2}p-p+{b}^{2}-1\right) }{{\left( {a}^{4}{p}^{4}-2{a}^{2}{b}^{2}{p}^{2}+2{a}^{2}{p}^{2}+{b}^{4}+2{b}^{2}-3\right) }^{2}}\\
\dfrac{\partial V_p^2}{\partial b}&=
\frac{768\pi^2b\left( {a}^{2}{p}^{3}-{a}^{2}{p}^{2}-{b}^{2}p-p+{b}^{2}-1\right) \left( 2{a}^{2}{p}^{3}-{a}^{2}{p}^{2}-2p-{b}^{2}+1\right) }{p{\left( {a}^{4}{p}^{4}-2{a}^{2}{b}^{2}{p}^{2}+2{a}^{2}{p}^{2}+{b}^{4}+2{b}^{2}-3\right) }^{2}}
\end{align*}
Both of the above vanish only when either of following holds.
\begin{align*}
&[a=\frac{\sqrt{b^2+\frac{p+1}{p-1}}}{p}],\ \ 
[a=-\frac{\sqrt{b^2+\frac{p+1}{p-1}}}{p}],\ \ [a=0,b=0],\\
&[a=0,b=-\sqrt{1-2\,p}],\ \ [a=0,b=\sqrt{1-2\,p}],\\
&[a=0,b=-\sqrt{-\frac{p}{p-1}-\frac{1}{p-1}}],\ \ 
[a=0,b=\sqrt{-\frac{p}{p-1}-\frac{1}{p-1}}],\\
&[a=-\frac{\sqrt{p-2}}{p^{\frac32}},b=0],\ \ 
[a=\frac{\sqrt{p-2}}{p^{\frac32}},b=0],\\
&[a=-\frac{\sqrt{1-\frac{2}{p}}}{p-1},b=-\frac{\sqrt{1-2\,p}}{p-1}],\ \ 
[a=\frac{\sqrt{1-\frac{2}{p}}}{p-1},b=-\frac{\sqrt{1-2\,p}}{p-1}],\\
&[a=-\frac{\sqrt{1-\frac{2}{p}}}{p-1},b=\frac{\sqrt{1-2\,p}}{p-1}],\ \ 
[a=\frac{\sqrt{1-\frac{2}{p}}}{p-1},b=\frac{\sqrt{1-2\,p}}{p-1}].
\end{align*}

\noindent
{\bf Case $\boldsymbol{p=1}$}\\
$\dfrac{\partial V_1^2}{\partial a}=\dfrac{\partial V_1^2}{\partial b}=0$ if and only if
$b(a^2-b^2-1)=a(a^2-b^2+1)=0$, that is
$[a=0,b=0]$. 

\medskip

\noindent
{\bf Case $\boldsymbol{p>1}$}\\
The real solutions of $\dfrac{\partial V_p^2}{\partial a}
=\dfrac{\partial V_p^2}{\partial b}=0$ are
\begin{align*}
&[a=\frac{\sqrt{b^2+\frac{p+1}{p-1}}}{p}],\ \ 
[a=-\frac{\sqrt{b^2+\frac{p+1}{p-1}}}{p}],\ \ [a=0,b=0],\\
&[a=-\frac{\sqrt{p-2}}{p^{\frac32}}, b=0],\ \ 
[a=\frac{\sqrt{p-2}}{p^{\frac32}}, b=0],
\end{align*}
where the last two solutions appear only when $p>2$.

\medskip

\begin{itemize}
\item $[a=0,b=0]$：In this cae, $(0,0,1/2)\in \mathcal P_p(1)$. 
\item $[a=\frac{\sqrt{b^2+\frac{p+1}{p-1}}}{p}]$：In this case, we have
$$
c=\dfrac12(1-b-pa)=\frac12\left(1-b-\sqrt{b^2+\frac{p+1}{p-1}}\right).
$$
But $c>0$ when $b<\dfrac{1}{1-p}<0$.
However we get
$$b+c=\frac12\left(1+b-\sqrt{b^2+\frac{p+1}{p-1}}\right)<0.$$
Hence  $(a,b,c)\not\in \mathcal P_p(1)$.
\item $[a=-\frac{\sqrt{b^2+\frac{p+1}{p-1}}}{p}]$：In this case we have
$$
c=\dfrac12(1-b-pa)=\frac12\left(1-b+\sqrt{b^2+\frac{p+1}{p-1}}\right).
$$
$c>0$ when $b>\dfrac{1}{1-p}$.
However, we get 
$$pa+c=\frac12\left(1-b-\sqrt{b^2+\frac{p+1}{p-1}}\right)<0.$$
Hence
$(a,b,c)\not\in \mathcal P_p(1)$. 
\item $[a=-\frac{\sqrt{p-2}}{p^{\frac32}},b=0]$：In this case we have
\begin{eqnarray*}
c&=&\dfrac12(1-b-pa)=\dfrac12 \left(1+\sqrt{\dfrac{p-2}{p}}\right)>0,\\
pa+c&=&\dfrac12 \left(1-\sqrt{\dfrac{p-2}{p}}\right)>0.
\end{eqnarray*}
Hence $(a,b,c)\in \mathcal P_p(1)$. 
\item $[a=\frac{\sqrt{p-2}}{p^{\frac32}},b=0]$：In this case we have
\begin{eqnarray*}
c&=&\dfrac12(1-b-pa)=\dfrac12 \left(1-\sqrt{\dfrac{p-2}{p}}\right)>0,\\
pa+c&=&\dfrac12 \left(1+\sqrt{\dfrac{p-2}{p}}\right)>0.
\end{eqnarray*}
Hence $(a,b,c)\in \mathcal P_p(1)$. 
\end{itemize}
We summarize our results of this subsection 4.2 as follows.
\begin{prop}
For the toric K\"ahler surface corresponding to
$\Delta_p$, if $1\le p\le 2$ then the cKEM-Futaki invariant vanishes only when
$(a,b)=(0,0)$. 
If $p>2$ then the cKEM-Futaki invariant vanishes when
$(a,b)=(0,0),\ (\pm\frac{\sqrt{p-2}}{p^{\frac32}},0)$. In this case $p>2$,
 LeBrun (\cite{L1}, Theorem C) shows that 
the K\"ahler class $\Omega$ contains two distinct cKEM metrics, which are ambitoric in the sense of
\cite{ACG16}. 
\end{prop}
Our computation complements LeBrun's result in that there are no 
non-K\"ahler solution for $1 < p < 2$.
But this nonexistence result has been obtained by Apostolov-Maschler \cite{AM} by showing the non-vanishing
of the cKEM-Futaki invariant using computer-assisted calculation.

\subsection{The case of the one point blow up of $\boldsymbol{\mathbf{C}\mathbf P^2}$}
Let
$\Delta_p$ be the convex hull of $(0,0),(p,0),(p,1-p),(0,1),\ (0<p<1)$. 
An affine linear function 
$f = a\mu_1+b\mu_2+c$ is positive on $\Delta_p$
if and only if 
$$c,b+c,(1-p)b+pa+c,pa+c>0.$$ 
We put
\begin{align*}
&\mathcal P_p:=\{(a,b,c)\in \mathbf{R}^3| c,b+c,(1-p)b+pa+c,pa+c>0\},\\
&\mathcal P_p(1):=\mathcal P_p\cap \{(a,b,c)| (2-p)b+2pa+4c=1\}.
\end{align*}
For $(a,b,c)\in \mathcal P_p$, we define $s_p(a,b,c)$ and $v_p(a,b,c)$ by
\begin{eqnarray*}
&\ &\int_{\partial \Delta_p}\dfrac{d\sigma}{(a\mu_1+b\mu_2+c)^2}\\
&\ &=\frac{1}{c\left( b+c\right) }+\frac{p}{\left( b+c\right) \left((1-p) 
b+pa+c\right) }\\
&\ &\quad +\frac{1-p}{\left((1-p) b+pa+c\right) \left( pa+c\right) }+\frac{p}{\left( pa+c\right) c}\\
&\ &=\frac{\left( \left( 2a-b\right) c+{a}^{2}-{b}^{2}\right) {p}^{2}+\left( {c}^{2}+2ac+ab+{b}^{2}\right) p+2{c}^{2}+2bc}{c\left( b+c\right) \left( pa+c\right) \left( (1-p)b+pa+c\right) }\\
&\ &=\dfrac{s_p(a,b,c)}{c\left( b+c\right) \left( pa+c\right) \left( (1-p)b+pa+c\right) }
\end{eqnarray*}
\begin{eqnarray*}
&\ &\int_{\Delta_p}\dfrac{d\mu_1d\mu_2}{(a\mu_1+b\mu_2+c)^4}\\
&\ & =p\{ -a\left( a-b\right) \left( {c}^{2}-2ac+2bc-ab+{b}^{2}\right) {p}^{3}\\
&\ &\quad -2\left( 2ac-bc+ab\right) \left( {c}^{2}-2ac+2bc-ab+{b}^{2}\right)  {p}^{2}\\
&\ &\quad -\left( 3{c}^{4}-12a{c}^{3}+12b{c}^{3}-18ab{c}^{2}+12{b}^{2}{c}^{2}-8{a}b^2c
+4{b}^{3}c-{a}b^3\right) {p} \\
&\ &\quad + 2c\left( c+b\right) \left( 3{c}^{2}+3bc+{b}^{2}\right) )\}
/(6c^2\left( b+c\right) ^2\left( pa+c\right) ^2\left( (1-p)b+pa+c\right)^2) \\
&\ &=\dfrac{v_p(a,b,c)}{6c^2\left( b+c\right) ^2\left( pa+c\right) ^2\left( (1-p)b+pa+c\right)^2}
\end{eqnarray*}
Hence for $(a,b,c)\in\mathcal P_p(1)$, we have
\begin{align*}
&V_p(a,b) ^2\\
&:=
\dfrac{48\pi^2s_p(a,b,(1+(p-2)b-2pa)/4)^2}{v_p(a,b,(1+(p-2)b-2pa)/4)}\\
&=
96 \pi^2 (3{\left( 2a-b\right) }^{2}p^3-2( 4{a}^{2}-4ab+2a-5{b}^{2}-b)p^2
-(20b^2+1)p+8b^2-2 ) ^2\\
&/p({\left( 2a-b\right) }^{2}\left( 4{a}^{2}-4ab+5{b}^{2}\right)p^5\\
&\quad -2\left( 2a-b\right) \left( 8{a}^{3}-12a^2{b}+8{a}^{2}+6{a}b^2-8ab-{b}^{3}+6{b}^{2}\right) p^4\\
&\quad -2\left( 48{a}^{2}{b}^{2}-4{a}^{2}-48{a}b^3-16{a}b^2+4ab+4{b}^{4}+8{b}^{3}-3{b}^{2}\right) p^3\\
&\quad +4\left( 16{a}^{2}{b}^{2}-4{a}^{2}-16{a}b^3-8{a}b^2+4ab+2a-12{b}^{4}+4{b}^{3}-5{b}^{2}-b\right) p^2\\
&\quad +(80{b}^{4}+24{b}^{2}-3)p
-32b^4-16b^2+6 ).
\end{align*}
Then $\displaystyle{\dfrac{\partial V_p^2}{\partial a}=
\dfrac{\partial V_p^2}{\partial b}}=0$ only when
\begin{enumerate}
\item[(1)]\ $a=\frac{2\sqrt{-9{b}^{2}{p}^{3}+\left( 21{b}^{2}+1\right) {p}^{2}+\left( 1-16{b}^{2}\right) p+4{b}^{2}-1}+3b{p}^{2}+\left( 1-2b\right) p}{6{p}^{2}-4p}$,
\item[(2)]\ $a=-\frac{2\sqrt{-9{b}^{2}{p}^{3}+\left( 21{b}^{2}+1\right) {p}^{2}+\left( 1-16{b}^{2}\right) p+4{b}^{2}-1}-3b{p}^{2}+\left( 2b-1\right) p}{6{p}^{2}-4p}$,
\item[(3)]\ 
$a=-\frac{1}{3p^2-2p}\sqrt{
\frac{p^3-6p^2+4p}{p-2}}-\frac{1}{6p-4}\sqrt{\frac{5p-2}{p-2}}
+\frac{1}{6p-4},b=-\frac{1}{3p-2}\sqrt{\frac{5p-2}{p-2}}$,
\item[(4)]\ $a=\frac{1}{3p^2-2p}\sqrt{
\frac{p^3-6p^2+4p}{p-2}}-\frac{1}{6p-4}\sqrt{\frac{5p-2}{p-2}}
+\frac{1}{6p-4},b=-\frac{1}{3p-2}\sqrt{\frac{5p-2}{p-2}}$,
\item[(5)]\ $a=-\frac{1}{3p^2-2p}\sqrt{
\frac{p^3-6p^2+4p}{p-2}}+\frac{1}{6p-4}\sqrt{\frac{5p-2}{p-2}}
+\frac{1}{6p-4},b=\frac{1}{3p-2}\sqrt{\frac{5p-2}{p-2}}$,
\item[(6)]\ $a=\frac{1}{3p^2-2p}\sqrt{
\frac{p^3-6p^2+4p}{p-2}}+\frac{1}{6p-4}\sqrt{\frac{5p-2}{p-2}}
+\frac{1}{6p-4},b=\frac{1}{3p-2}\sqrt{\frac{5p-2}{p-2}}$,
\item[(7)]\ $a=\frac{p-1}{p^2},b=\frac{1}{p}$,
\item[(8)]\ $a=-\frac{1}{p^2},b=-\frac{1}{p}$,
\item[(9)]\ $a=-\frac{\sqrt{9{p}^{2}-8p}+p}{4{p}^{2}},b=0$,
\item[(10)]\ $a=\frac{\sqrt{9{p}^{2}-8p}-p}{4{p}^{2}},b=0$,
\item[(11)]\ $a=-\frac{1}{6p-4}\sqrt{\frac{p^2+p-1}{p-1}}+\frac{1}{6p-4},
b=-\frac{1}{3p-2}\sqrt{\frac{p^2+p-1}{p-1}}$,
\item[(12)]\ $a=\frac{1}{6p-4}\sqrt{\frac{p^2+p-1}{p-1}}+\frac{1}{6p-4},
b=\frac{1}{3p-2}\sqrt{\frac{p^2+p-1}{p-1}}$,
\item[(13)]\ $a=-\frac{\sqrt{{p}^{4}-4{p}^{3}+16{p}^{2}-16p+4}-{p}^{2}+4p-2}{2{p}^{3}-4{p}^{2}+12p-8},
b=-\frac{\sqrt{{p}^{4}-4{p}^{3}+16{p}^{2}-16p+4}}{{p}^{3}-2{p}^{2}+6p-4}$,
\item[(14)]\ $a=\frac{\sqrt{{p}^{4}-4{p}^{3}+16{p}^{2}-16p+4}+{p}^{2}-4p+2}{2{p}^{3}-4{p}^{2}+12p-8},
b=\frac{\sqrt{{p}^{4}-4{p}^{3}+16{p}^{2}-16p+4}}{{p}^{3}-2{p}^{2}+6p-4}$,
\item[(15)]\ $a=-\frac{-p+2\sqrt{1-p}+2}{2{p}^{2}},b=0$,
\item[(16)]\ $a=\frac{p+2\sqrt{1-p}-2}{2{p}^{2}},b=0.$
\end{enumerate}
Up to this point in this subsection, we used Maxima to derive the above conclusions.
Among the above, real solutions are the following:

$(1),\ (2),\ (7),\ (8),\ (15),\ (16)$, 

$(9),\ (10)$：when $1>p\ge 8/9$

$(11),\ (12)$：when $0<p<(\sqrt{5}-1)/2$

$(13),\ (14)$：when $0<p\le \alpha,\ \beta\le p<1$, where $0<\alpha<\beta <1$
are the real roots of 
$p^4-4p^3+16p^2-16p+4=0$.\\
%\vspace{4mm}
Let us check the cases (7) - (16) whether $(a,b,c) \in \mathcal P_p(1)$ or not. 
The proof of each case is elementary.
We only give a detailed proof only for (11) and (13) for the reader's convenience.
We leave the cases (1) and (2) to later study. 

\begin{enumerate}
\item[(7)]：$(a,b,c)\not\in \mathcal P_p(1)$
since $c = \dfrac{1+(p-2)b-2pa}{4}=0$.

\item[(8)]：$(a,b,c)\not\in \mathcal P_p(1)$
since $b+c=0$.

\item[(9)]：$(a,b,c)\in \mathcal P_p(1)$
since 
$$c = b+c=\frac38+\frac{\sqrt{9p^2-8p}}{8p}>0$$ 
and 
$$(1-p)b+pa+c=pa+c=\frac18-\frac{\sqrt{9p^2-8p}}{8p}>0.$$

\item[(10)]：$(a,b,c)\in \mathcal P_p(1)$ since 
$$c = b+c=\frac38-\frac{\sqrt{9p^2-8p}}{8p}>0$$
and
$$(1-p)b+pa+c=pa+c=\frac18+\frac{\sqrt{9p^2-8p}}{8p}>0.$$

\item[(11)]：$(a,b,c)\not\in \mathcal P_p(1)$ 
\begin{proof}
First of all, recall $0<p<(\sqrt{5}-1)/2$. We need to check the signs of
$$c =\frac{\sqrt{\frac{{p}^{2}+p-1}{p-1}}+p-1}{6p-4},$$
$$b+c=-\frac{\sqrt{\frac{{p}^{2}+p-1}{p-1}}-p+1}{6p-4},$$
$$(1-p)b+pa+c=\frac{\left( p-1\right) \sqrt{\frac{{p}^{2}+p-1}{p-1}}+2p-1}{6p-4}$$
and 
$$pa+c=-\frac{\left( p-1\right) \sqrt{\frac{{p}^{2}+p-1}{p-1}}-2p+1}{6p-4}.$$
Since $4-6p > 0$, 
it is sufficient to prove that
\begin{equation}\label{32}
(4-6p)(pa+c)=(p-1)\sqrt{\dfrac{p^2+p-1}{p-1}}-2p+1<0.
\end{equation}
When $\frac12 \le p<\dfrac{\sqrt{5}-1}{2}$, \eqref{32} is trivial.
When $0<p<\frac12$,
\begin{equation*}
\begin{split}
(p-1)\sqrt{\dfrac{p^2+p-1}{p-1}}<2p-1
&\iff
\sqrt{\dfrac{p^2+p-1}{p-1}}>\dfrac{2p-1}{p-1}\\
&\iff
\dfrac{p^2+p-1}{p-1}>\dfrac{(2p-1)^2}{(p-1)^2}\\
&\iff
p^2-4p+2>0
\end{split}
\end{equation*}
This completes the proof.
\end{proof}

\item[(12)]：$(a,b,c)\not\in \mathcal P_p(1)$ since
$$c =-\frac{\sqrt{\frac{{p}^{2}+p-1}{p-1}}-p+1}{6p-4},$$
$$b+c=\frac{\sqrt{\frac{{p}^{2}+p-1}{p-1}}+p-1}{6p-4},$$
$$(1-p)b+pa+c=-\frac{\left( p-1\right) \sqrt{\frac{{p}^{2}+p-1}{p-1}}-2p+1}{6p-4} < 0$$
and
$$pa+c=\frac{\left( p-1\right) \sqrt{\frac{{p}^{2}+p-1}{p-1}}+2p-1}{6p-4}.$$

\item[(13)]：If $0<p\le \alpha$ then $(a,b,c)\in \mathcal P_p(1)$. If  
$\beta\le p<1$, then $(a,b,c)\not\in \mathcal P_p(1)$. This is because
\begin{equation*}
\begin{split}
&c=\dfrac{\sqrt{p^4-4p^3+16p^2-16p+4}+p^2+2p-2}{2(p^3-2p^2+6p-4)}\\
&b+c=-\dfrac{\sqrt{p^4-4p^3+16p^2-16p+4}-p^2-2p+2}{2(p^3-2p^2+6p-4)}\\
&(1-p)b+pa+c=\dfrac{(p-1)(\sqrt{p^4-4p^3+16p^2-16p+4}+p^2-2p+2)}{2(p^3-2p^2+6p-4)}\\
&pa+c=-\dfrac{(p-1)(\sqrt{p^4-4p^3+16p^2-16p+4}-p^2+2p-2)}{2(p^3-2p^2+6p-4)}
\end{split}
\end{equation*}
When $0<p\le \alpha$, then $c,b+c,(1-p)b+pa+c$ and $pa+c$ are positive.
On the other hand, when $\beta\le p<1$, we have $(1-p)b+pa+c<0$. 
\begin{proof}
For $0<p\le \alpha=0.386\cdots$,
$(a,b,c)\in \mathcal P_p(1)$

\begin{equation*}
\begin{split}
c
&=
\dfrac{\sqrt{p^4-4p^3+16p^2-16p+4}+p^2+2p-2}{2(p^3-2p^2+6p-4)}\\
b+c
&=
\dfrac{-\sqrt{p^4-4p^3+16p^2-16p+4}+p^2+2p-2}{2(p^3-2p^2+6p-4)}\\
(1-p)b+pa+c
&=(1-p)
\dfrac{-\sqrt{p^4-4p^3+16p^2-16p+4}-p^2+2p-2}{2(p^3-2p^2+6p-4)}\\
pa+c
&=
(1-p)
\dfrac{\sqrt{p^4-4p^3+16p^2-16p+4}-p^2+2p-2}{2(p^3-2p^2+6p-4)}\\
\end{split}
\end{equation*}
It is easy to see that
$p^3-2p^2+6p-4$ is negative on $(0,\alpha]$.
So, to prove $c,b+c,(1-p)b+pa+c, pa+c>0$,
it is sufficient to see that
\begin{equation*}
\begin{split}
A&:=
2(p^3-2p^2+6p-4)c\\
&=
\sqrt{p^4-4p^3+16p^2-16p+4}+p^2+2p-2<0\\
B&:=
2(p^3-2p^2+6p-4)(b+c)\\
&=
-\sqrt{p^4-4p^3+16p^2-16p+4}+p^2+2p-2<0\\
C&:=
\dfrac{2(p^3-2p^2+6p-4)((1-p)b+pa+c)}{1-p}\\
&=
-\sqrt{p^4-4p^3+16p^2-16p+4}-p^2+2p-2<0\\
D&:=
\dfrac{2(p^3-2p^2+6p-4)(pa+c)}{1-p}\\
&=
\sqrt{p^4-4p^3+16p^2-16p+4}-p^2+2p-2<0.
\end{split}
\end{equation*}
Since $A\ge B$ and $A\ge D\ge C$,
it sufficient to see $A<0$.
For $0<p\le \alpha$, $p^4-4p^3+16p^2-16p+4,-p^2-2p+2>0$.
Therefore
$$
A<0
\iff
p^4-4p^3+16p^2-16p+4<(-p^2-2p+2)^2
\iff
-8(p-1)^2p<0
$$
Thus we are done for $0<p\le \alpha$.

For $\beta=0.844\cdots \le p<1$, 
it is easy to see that $p^3-2p^2+6p-4$ and $p^2-2p+2$ are positive.
Hence
$$
(1-p)b+pa+c
=\dfrac{(p-1)(\sqrt{p^4-4p^3+16p^2-16p+4}+p^2-2p+2)}
{2(p^3-2p^2+6p-4)}<0.
$$
Thus we are done for $\beta \le p<1$.
\end{proof}

\item[(14)]：If $0<p\le \alpha$, then $(a,b,c)\in \mathcal P_p(1)$. 
If $\beta\le p<1$, then $(a,b,c)\not\in \mathcal P_p(1)$.
\begin{equation*}
\begin{split}
&c=-\dfrac{\sqrt{p^4-4p^3+16p^2-16p+4}-p^2-2p+2}{2(p^3-2p^2+6p-4)}\\
&b+c=\dfrac{\sqrt{p^4-4p^3+16p^2-16p+4}+p^2+2p-2}{2(p^3-2p^2+6p-4)}\\
&(1-p)b+pa+c=-\dfrac{(p-1)(\sqrt{p^4-4p^3+16p^2-16p+4}-p^2+2p-2)}{2(p^3-2p^2+6p-4)}\\
&pa+c=\dfrac{(p-1)(\sqrt{p^4-4p^3+16p^2-16p+4}+p^2-2p+2)}{2(p^3-2p^2+6p-4)}
\end{split}
\end{equation*}
When $0<p\le \alpha$ $\beta\le p<1$, $c,b+c,(1-p)b+pa+c$ and $pa+c$ are 
positive.
On the other hand, when $\beta\le p<1$, we have
$(1-p)b+pa+c<0$. 
%\vspace{4mm}
\item[(15)]：$(a,b,c)\not\in \mathcal P_p(1)$
\begin{align*}
&c=b+c= \frac{\sqrt{1-p}+1}{2p},\\
&pa+c=(1-p)b+pa+c=-\frac{-p+\sqrt{1-p}+1}{2p}<0
\end{align*}
\item[(16)]：$(a,b,c)\in \mathcal P_p(1)$
\begin{align*}
&c=b+c=-\frac{\sqrt{1-p}-1}{2p}>0,\\
&(1-p)b+pa+c=pa+c=\frac{p+\sqrt{1-p}-1}{2p}>0.
\end{align*}
\end{enumerate}
%\vspace{4mm}
We record the data of the cases (1), (2) for later study.\\
\vspace{4mm}
(1)：
\begin{align*}
&c=-\frac{\sqrt{-9{b}^{2}{p}^{3}+\left( 21{b}^{2}+1\right) {p}^{2}+\left( 1-16{b}^{2}\right) p+4{b}^{2}-1}+\left( 3b-1\right) p-2b+1}{6p-4}\\
&b+c=\\
&-\frac{\sqrt{-9{b}^{2}{p}^{3}+\left( 21{b}^{2}+1\right) {p}^{2}+\left( 1-16{b}^{2}\right) p+4{b}^{2}-1}+\left( -3b-1\right) p+2b+1}{6p-4}\\
&(1-p)b+pa+c=\\
&\frac{\sqrt{-9{b}^{2}{p}^{3}+\left( 21{b}^{2}+1\right) {p}^{2}+\left( 1-16{b}^{2}\right) p+4{b}^{2}-1}+\left( 3b+3\right) p-2b-1}{6p-4}\\
&pa+c=\\
&\frac{\sqrt{-9{b}^{2}{p}^{3}+\left( 21{b}^{2}+1\right) {p}^{2}+\left( 1-16{b}^{2}\right) p+4{b}^{2}-1}+6b{p}^{2}+\left( 3-7b\right) p+2b-1}{6p-4}
\end{align*}
(2)：
\begin{align*}
&c=\frac{\sqrt{-9{b}^{2}{p}^{3}+\left( 21{b}^{2}+1\right) {p}^{2}+\left( 1-16{b}^{2}\right) p+4{b}^{2}-1}+\left( 1-3b\right) p+2b-1}{6p-4}\\
&b+c=\\
&\frac{\sqrt{-9{b}^{2}{p}^{3}+\left( 21{b}^{2}+1\right) {p}^{2}+\left( 1-16{b}^{2}\right) p+4{b}^{2}-1}+\left( 3b+1\right) p-2b-1}{6p-4}\\
&(1-p)b+pa+c=\\
&-\frac{\sqrt{-9{b}^{2}{p}^{3}+\left( 21{b}^{2}+1\right) {p}^{2}+\left( 1-16{b}^{2}\right) p+4{b}^{2}-1}+3b{p}^{2}+\left( -5b-2\right) p+2b+1}{6p-4}\\
&pa+c=\\
&-\frac{\sqrt{-9{b}^{2}{p}^{3}+\left( 21{b}^{2}+1\right) {p}^{2}+\left( 1-16{b}^{2}\right) p+4{b}^{2}-1}-3b{p}^{2}+\left( 5b-2\right) p-2b+1}{6p-4}
\end{align*}
\vspace{10mm}

To sum up, leaving (1), (2) aside,
if $0<p<\alpha$ then the cKEM-Futaki invariant vanishes
for $(13),\ (14),$ and $(16)$.
If $\alpha\le p\le 8/9 $ then the cKEM-Futaki invariant 
vanishes only for $(16)$.
If $8/9<p<1$ then cKEM-Futaki invariant 
vanishes for $(9),\ (10)$ and $(16)$. We wish to compare this with the following
result of LeBrun.
\begin{thm}[LeBrun \cite{L2}]\label{LCMP16}
Let $M$ be the blow-up of $\mathbf C\mathbf P^2$ at one point.\\
$\mathrm{(a)}$\ For any K\"ahler class, there exists a K\"ahler metric which is conformal to a 
conformally K\"ahler, Einstein-Maxwell metric.\\
$\mathrm{(b)}$\ Express an arbitrary K\"ahler class as 
$\Omega = u\mathcal L - v\mathcal E$ where $\mathcal L$ and $\mathcal E$ are the Poincar\'e duals of a projective line and the exceptional divisor. 
If $9 < u/v $ then there are two K\"ahler metrics which are conformal to a 
conformally K\"ahler, Einstein-Maxwell metric. One of these metrics $g$ has two positive
potential functions $f$ of Hamiltonian Killing vector fields such that $\Tilde g = f^{-2}g$ is
a conformally K\"ahler, Einstein-Maxwell metric. 
Further there is an orientation reversing isometry between these two conformally K\"ahler, Einstein-Maxwell metrics $\Tilde g$.\\
All the K\"ahler metrics and positive
potential functions of Hamiltonian Killing vector fields in $\mathrm{(a)}$ and $\mathrm{(b)}$ are 
$U(2)$-invariant, and there are no other $U(2)$-invariant conformally K\"ahler, Einstein-Maxwell metrics.
\end{thm}

\begin{thm}\label{OPB} Let $M$ be the blow-up of $\mathbf C\mathbf P^2$ at one point, 
$\Delta_p$ the convex hull of $(0,0),(p,0),(p,1-p),(0,1),\ (0<p<1)$ in $(\mu_1,\mu_2)$-plane,
and consider $\Delta_p$ as the moment map image of $M$. Let $0<\alpha<\beta <1$
be the real roots of 
$$p^4-4p^3+16p^2-16p+4=0.$$ 
\noindent
$\mathrm{(a)}$\ For $0 < p < 1$, the affine function 
$$ f = \frac{p+2\sqrt{1-p}-2}{2{p}^{2}} \mu_1 -\frac{\sqrt{1-p}-1}{2p} $$
corresponds to the conformally K\"ahler, Einstein-Maxwell metric in $\mathrm{(a)}$
in Theorem \ref{LCMP16}.\\
$\mathrm{(b)}$\  For $\frac89 < p < 1$, the two affine functions 
$$ f = -\frac{\sqrt{9{p}^{2}-8p}+p}{4{p}^{2}} \mu_1 + \frac38+\frac{\sqrt{9p^2-8p}}{8p},$$
$$ f = \frac{\sqrt{9{p}^{2}-8p}-p}{4{p}^{2}}\ \mu_1 + \frac38-\frac{\sqrt{9p^2-8p}}{8p}$$
correspond to the conformally K\"ahler, Einstein-Maxwell metric in $\mathrm{(b)}$
in Theorem \ref{LCMP16}.\\
$\mathrm{(c)}$\ For $0 < p < \alpha$, the two affine functions
\begin{eqnarray*}
f &=& -\frac{\sqrt{{p}^{4}-4{p}^{3}+16{p}^{2}-16p+4}-{p}^{2}+4p-2}{2{p}^{3}-4{p}^{2}+12p-8}\mu_1\\
&&- \frac{\sqrt{{p}^{4}-4{p}^{3}+16{p}^{2}-16p+4}}{{p}^{3}-2{p}^{2}+6p-4}\mu_2\\
&&+ \dfrac{\sqrt{p^4-4p^3+16p^2-16p+4}+p^2+2p-2}{2(p^3-2p^2+6p-4)},
\end{eqnarray*}
\begin{eqnarray*}
f &=& \frac{\sqrt{{p}^{4}-4{p}^{3}+16{p}^{2}-16p+4}+{p}^{2}-4p+2}{2{p}^{3}-4{p}^{2}+12p-8}\mu_1\\
&&+ \frac{\sqrt{{p}^{4}-4{p}^{3}+16{p}^{2}-16p+4}}{{p}^{3}-2{p}^{2}+6p-4} \mu_2 \\
&&- \dfrac{\sqrt{p^4-4p^3+16p^2-16p+4}-p^2-2p+2}{2(p^3-2p^2+6p-4)}
\end{eqnarray*}
are positive and satisfy $Fut_f = 0$. If there is a K\"ahler metric $g$ such that $f^{-2}g$ gives
a conformally K\"ahler, Einstein-Maxwell metric then it has $U(1)\times U(1)$-symmetry.
\end{thm}
\begin{proof}
Recall that, in the classification (1) - (16), the cases with $Fut_f = 0$ are
the cases (1), (2), (9), (10), (13), (14) and (16). 
A toric cKEM metric in these cases, if any, has $U(2)$-symmetry if and only if
$b=0$. 

We see that (1) with $b=0$ or (2) with $b=0$ do not occur.
In fact, if (1) with $b=0$ occurs then
$$ a = \frac{2\sqrt{{p}^{2}+p-1}+ p}{6{p}^{2}-4p}, \qquad 4c = \frac{2p-2 - \sqrt{2p^2+p-1}}{3p-2}.$$
We have to have ${p}^{2}+p-1>0$, and thus we have only to consider the case
$(\sqrt{5} -1)/2 < p < 1$.
In this range, we have 
$$
\dfrac{pa+c}{c}=
\dfrac{1+2pa}{1-2pa}
=
\dfrac{-\sqrt{p^2+p-1}-2p+1}{\sqrt{p^2+p-1}+1-p}<0.
$$
Hence $c$ or $pa+c$ is negative.
So $f(0,0) < 0$ or $f(p,0) < 0$.
If (2) with $b=0$ occurs then 
$$ a = - \frac{2\sqrt{{p}^{2}+p-1}-p}{6{p}^{2}-4p}, \qquad 
pa + c = \frac{-\sqrt{{p}^{2}+p-1}+2p-1}{6p-4}.$$
For $(\sqrt{5} -1)/2 < p < 1$ we have $pa + c  < 0$ since the numerator and denominator
both change sign at $p=2/3$ and $pa + c = -3/4$ at $p = 2/3$. So $f(p,0) < 0$.

It follows that $U(2)$-symmetry occurs exactly when (9), (10), (16) 
because we have $b=0$ in the cases (9), (10) and (16).

The moment map image $\Delta_p$ determines the K\"ahler class 
$\Omega = u\mathcal L - v \mathcal E$ with $u = 1$ and $v=1-p$.
Thus $u/v \le 9$ if and only if $p \le 8/9$. 
In this region, only the case (16) allows an $f$ with vanishing cKEM-Futaki invariant, and in fact
Theorem \ref{LCMP16} shows there is one cKEM metric with $U(2)$-symmetry.
Moreover by Theorem 3 in \cite{AM}, for a given $f$, a toric K\"ahler metric $g$ such that
$f^{-2}g$ is a cKEM metric is unique. Thus (a) holds.

In the region $u/v > 9$, that is, $p > 8/9$, the case (9), (10) and (16) gives an $f$ with 
vanishing cKEM-Futaki invariant. By the similar arguments as in the case of $p \le 8/9$,
these three cases correspond to the three LeBrun solutions in Theorem \ref{LCMP16}
cited above. Moreover the cases (9) and (10) correspond to (b) in Theorem \ref{LCMP16}, which can
be checked by computing $V_p(a,0)^2$ in (\ref{eq:6.6}). 
Put 
\begin{eqnarray*}
f(a,p)&:=&\dfrac{V_p(a,0)^2}{96\pi^2}\\
&=&\frac{(12a^2p^3-4a(2a+1)p^2-p-2)^2}{p(16a^4p^5-32a^3(a+1)p^4
+8a^2p^3+8a(1-2a)p^2-3p+6)}.
\end{eqnarray*}
Then we have 
$$f(-\frac{\sqrt{9p^2-8p}+p}{4p^2},p)=
f(\frac{\sqrt{9p^2-8p}-p}{4p^2},p)=5-\frac{2}{p},$$
which shows the solutions corresponding to (9) and (10) are homothetic.
(These are not isometric since the total scalar curvature and the volume have different 
values.)
But for the case (16) the value of $f(a,p)$ is not equal to the cases (9) and (10) because
\begin{eqnarray*}
&&f(\frac{p+2\sqrt{1-p}-2}{2p^2},p)\\
&&=
-\frac{4{p}^{4}+\sqrt{1-p}\left( 24{p}^{3}-112{p}^{2}+112p-32\right) -68{p}^{3}+164{p}^{2}-128p+32}{{p}^{4}+6{p}^{3}+\sqrt{1-p}\left( 16{p}^{2}-16p\right) -24{p}^{2}+16p}
\end{eqnarray*}
and
\begin{eqnarray*}
&&f(-\frac{\sqrt{9p^2-8p}+p}{4p^2},p)-
f(\frac{p+2\sqrt{1-p}-2}{2p^2},p)\\
&&=\frac{9{p}^{3}+\sqrt{1-p}\left( 24{p}^{2}-32p\right) -40{p}^{2}+32p}{{p}^{3}+6{p}^{2}+\sqrt{1-p}\left( 16p-16\right) -24p+16}.
\end{eqnarray*}
Hence we have proved (b).

The statement (c) is the possibility of the cases (13) and (14).
This completes the proof of Theorem \ref{OPB}.
\end{proof}

We have not been able to construct a cKEM metric for the cases (13) and (14).
There is an ansatz to construct local and global ambitoric 
solutions, see \cite{ACG16}, \cite{ACG15}, \cite{AM}.
We have not been able to rule out the cases (1) and (2).
We leave these problems to the interested readers.

\subsection{Hirzebruch surfaces}
Let
$\Delta_{p,q}$ be the convex hull of $(0,0),(p,0),(p,(1-p)q),(0,q),
\ (0<p<1, q\in \mathbf{N})$. 
An affine linear function
$f = a\mu_1+b\mu_2+c$ is positive on
$\Delta_{p,q}$ if and only if $c,qb+c,(1-p)qb+pa+c,pa+c>0$. 
We put
\begin{align*}
&\mathcal P_{p,q}:=\{(a,b,c)\in \mathbf{R}^3| c,qb+c,(1-p)qb+pa+c,pa+c>0\},\\
&\mathcal P_{p,q}(1):=\mathcal P_p\cap \{(a,b,c)| (2-p)qb+2pa+4c=1\}.
\end{align*}
For $(a,b,c)\in \mathcal P_{p,q}$, we have
\begin{align*}
&\int_{\partial \Delta_{p,q}}\dfrac{d\sigma}{(a\mu_1+b\mu_2+c)^2}\\
&=\frac{q}{c\left( qb+c\right) }+\frac{p}{\left( qb+c\right) \left( (1-p)qb+pa+c
\right) }\\
&+
\frac{\left( 1-p\right) q}{\left( pa+c\right) \left( (1-p)qb+pa+c\right) }
+
\frac{p}{c\left( pa+c\right) }
\\
&=-\{(\left( ab+{b}^{2}\right) {q}^{2}+\left( bc-{a}^{2}-ab\right) q-2ac)p^2\\
&+(\left( 2bc-ab-{b}^{2}\right) {q}^{2}+\left( {c}^{2}-2\left( a+b\right) c\right) q-2{c}^{2})p
-2cq\left( qb+c\right) \}\\
&/\{\left( pa+c\right) \left( \left( 1-p\right) qb+pa+c\right) \left( qb+c\right) c\}\\
&=\dfrac{s_{p,q}(a,b,c)}{\left( pa+c\right) \left( \left( 1-p\right) qb+pa+c\right) \left( qb+c\right) c}
\end{align*}
\begin{align*}
&\int_{\Delta_p}\dfrac{d\mu_1d\mu_2}{(a\mu_1+b\mu_2+c)^4}\\
&=\{pq({a}b^3{p}^{3}{q}^{3}+2{b}^{3}c{p}^{2}{q}^{3}-2{a}b^3{p}^{2}{q}^{3} 
-4 {b}^{3}cp{q}^{3}+{a}b^3p{q}^{3}+2{b}^{3}c{q}^{3}\\
&+2{a}b^2c{p}^{3}{q}^{2} -2 {a}^{2}{b}^{2}{p}^{3}{q}^{2}+4{b}^{2}{c}^{2}{p}^{2}{q}^{2}
-10 {a}b^2c{p}^{2}{q}^{2}+2{a}^{2}{b}^{2}{p}^{2}{q}^{2}\\
&-12 {b}^{2}{c}^{2}p{q}^{2}+8{a}b^2cp{q}^{2}+8{b}^{2}{c}^{2}{q}^{2}
+ab{c}^{2}{p}^{3}q-4a^2{b}c{p}^{3}q+a^3{b}{p}^{3}q\\
&+2b{c}^{3}{p}^{2}q-14ab{c}^{2}{p}^{2}q+8a^2{b}c{p}^{2}q 
-12 b{c}^{3}pq+18ab{c}^{2}pq+12b{c}^{3}q\\
&-{a}^{2}{c}^{2}{p}^{3}+2{a}^{3}c{p}^{3} -4 a{c}^{3}{p}^{2}
+8{a}^{2}{c}^{2}{p}^{2} -3 {c}^{4}p+12a{c}^{3}p+6c^{4})\}\\
&/\{6{c}^{2}{\left( pa+c\right) }^{2}{\left( qb+c\right) }^{2}{\left( (1-p)qb+pa+c\right) }^{2}\}\\
&=
\dfrac{v_{p,q}(a,b,c)}{6{c}^{2}{\left( pa+c\right) }^{2}{\left( qb+c\right) }^{2}{\left( (1-p)qb+pa+c\right) }^{2}}
\end{align*}
Hence for $(a,b,c)\in \mathcal P_{p,q}(1)$, we have
\begin{align*}
&V_{p,q}(a,b)^2&\\
&:=
\dfrac{48\pi^2s_{p,q}(a,b,(1+(p-2)qb-2pa)/4)^2}{v_{p,q}(a,b,(1+(p-2)qb-2pa)/4)}\\
&=
\{6({b}^{2}{p}^{3}{q}^{3}+2{b}^{2}{p}^{2}{q}^{3} -12{b}^{2}p{q}^{3}+8{b}^{2}{q}^{3} -4
ab{p}^{3}{q}^{2}+2{b}^{2}{p}^{3}{q}^{2}\\
&+8ab{p}^{2}{q}^{2}+8{b}^{2}{p}^{2}{q}^{2}+2b{p}^{2}{q}^{2} -8 {b}^{2}p{q}^{2}
+4a^{2}{p}^{3}q-8ab{p}^{3}q\\
&-8{a}^{2}{p}^{2}q-4 a{p}^{2}q+pq-2q+8{a}^{2}{p}^{3}-2p)^2\}\\
&/
\{
pq
(
5{b}^{4}{p}^{5}{q}^{4}-2{b}^{4}{p}^{4}{q}^{4} -8{b}^{4}{p}^{3}{q}^{4}
-48{b}^{4}{p}^{2}{q}^{4}+80{b}^{4}p{q}^{4}-32 {b}^{4}{q}^{4}\\
&-24 {a}b^3{p}^{5}{q}^{3}+16{a}b^3{p}^{4}{q}^{3}+12{b}^{3}{p}^{4}{q}^{3}
+96{a}b^3{p}^{3}{q}^{3}-16{b}^{3}{p}^{3}{q}^{3}-64{a}b^3{p}^{2}{q}^{3}\\
&+16{b}^{3}{p}^{2}{q}^{3}+40{a}^{2}{b}^{2}{p}^{5}{q}^{2}-48 {a}^{2}{b}^{2}{p}^{4}{q}^{2}
-40{a}b^2{p}^{4}{q}^{2}-96{a}^{2}{b}^{2}{p}^{3}{q}^{2}+32{a}b^2{p}^{3}{q}^{2}\\
&+6{b}^{2}{p}^{3}{q}^{2}+64{a}^{2}{b}^{2}{p}^{2}{q}^{2}-32{a}b^2{p}^{2}{q}^{2}
-20{b}^{2}{p}^{2}{q}^{2}+24{b}^{2}p{q}^{2}-16 {b}^{2}{q}^{2}\\
&-32a^3{b}{p}^{5}q+64a^3{b}{p}^{4}q+48a^2{b}{p}^{4}q
-8 ab{p}^{3}q+16ab{p}^{2}q-4b{p}^{2}q+16{a}^{4}{p}^{5}\\
&-32 {a}^{4}{p}^{4}-32{a}^{3}{p}^{4}+8{a}^{2}{p}^{3}
-16 {a}^{2}{p}^{2}+8a{p}^{2}-3 p+6
)
\}
\end{align*}

For example, the following
are real solutions of
$\displaystyle{
\dfrac{\partial V_{p,q}^2}{\partial a}=
\dfrac{\partial V_{p,q}^2}{\partial b}=0}$:

\begin{enumerate}
\item[(1)]\ $[a=\frac{p+2\sqrt{1-p}-2}{2{p}^{2}},b=0]$,
\item[(2)]\ $[a=\frac{\pm\sqrt{p(pq^2+4q(p-2)-4p)}-pq}{4{p}^{2}},
b=0],$
\item[(3)]\ $[a=-\frac{\sqrt{4(1-p)^2q^2-4(p-1)(p-2)pq+p^4}-2\left( p-1\right) q+p(p-2)}{2((2(p-1)(p-2)q-p^3))},\\
\ \ b=-\frac{\sqrt{4(1-p)^2q^2-4(p-1)(p-2)pq+p^4}}{q(2(p-1)(p-2)q-p^3)}],$
\item[(4)]\ $[a=\frac{\sqrt{4(1-p)^2q^2-4(p-1)(p-2)pq+p^4}+2( p-1) q-p(p-2)}{2((2(p-1)(p-2)q-p^3))},\\
\ \ b=\frac{\sqrt{4(1-p)^2q^2-4(p-1)(p-2)pq+p^4}}{q((2(p-1)(p-2)q-p^3))}]$
\end{enumerate}

\vspace{4mm}

Let us check these four cases.

\vspace{4mm}

 $(1)\ [a=\frac{p+2\sqrt{1-p}-2}{2{p}^{2}},b=0]$：$(a,b,c)\in \mathcal P_{p,q}(1)$

\vspace{2mm}

In this case, $c=qb+c,pa+c=(1-p)qb+pa+c$ are independent of 
$q$. Hence these are positive by the computation when $q=1$.

\vspace{4mm}

$(2)\ [a=\frac{\pm\sqrt{p(pq^2+4q(p-2)-4p)}-pq}{4{p}^{2}},b=0]$：$(a,b,c)\not\in \mathcal P_{p,q}(1)$

\vspace{2mm}

When $q=2$, these are not real solutions. When $q\ge 3$
and 
$$pq^2+4q(p-2)-4p>0,$$
we have
\begin{equation*}
pa+c=
\frac{(2-q)p
\pm\sqrt{p(pq^2+4q(p-2)-4p)}}{8p}<0.
\end{equation*}

%\vspace{2mm}

\begin{align*}
&(3)\ [a=-\frac{\sqrt{4(1-p)^2q^2-4(p-1)(p-2)pq+p^4}-2\left( p-1\right) q+p(p-2)}{2((2(p-1)(p-2)q-p^3))},\\
&\ \ \ \ \ \ b=-\frac{\sqrt{4(1-p)^2q^2-4(p-1)(p-2)pq+p^4}}{q(2(p-1)(p-2)q-p^3)}],\\
&(4)\ [a=\frac{\sqrt{4(1-p)^2q^2-4(p-1)(p-2)pq+p^4}+2( p-1) q-p(p-2)}{2((2(p-1)(p-2)q-p^3))},\\
&\ \ \ \ \ \ b=\frac{\sqrt{4(1-p)^2q^2-4(p-1)(p-2)pq+p^4}}{q((2(p-1)(p-2)q-p^3))}]
\end{align*}

\vspace{3mm}

Take $q=2,3,4$ and perform a numerical analysis, then we see
that there are two roots $0<\alpha_q<\beta_q<1$ of the quartic equation in $p$:
$$4(1-p)^2q^2-4(p-1)(p-2)pq+p^4 = 0,$$
and that for $0<p<\alpha_q$, we have $(a,b,c)\in \mathcal P_{p,q}(1)$ so that
the cKEM-Futaki invariant vanishes. 

We conclude this section with the following two remarks.
\begin{rem}
It is likely that the the case (1) corresponds to LeBrun's construction in \cite{L2}, Theorem D with
$k \ge 2$. This case should be the only case with $U(2)$-symmetry. 
We may prove it by showing $b = 0$ occurs only in the case (1).
\end{rem}
\begin{rem}
It would be interesting if one can prove or disprove the existence of cKEM metrics in the cases of (3) and (4) with $0<p<\alpha_q$ since, if any, the solutions necessarily have
$U(1)\times U(1)$-symmetry.
\end{rem}

\end{document}